 \documentclass[brochure,english,12pt]{smfbourbaki}

\usepackage[T1]{fontenc}
\usepackage{lmodern,amssymb,bm}
\usepackage{mathrsfs}

\usepackage{babel}

\usepackage[utf8]{inputenc}

\usepackage[colorlinks=true, linkcolor=blue, citecolor=red, urlcolor=blue]{hyperref}
\usepackage[
backend=biber,
style=authoryear, 
citestyle=authoryear-comp,
maxnames=7,
sortcites=false %%%% to keep the order in \cite command
]{biblatex}
\usepackage{csquotes}

\DeclareDelimFormat{nameyeardelim}{\addcomma\space}

\DeclareNameAlias{sortname}{given-family}

\renewcommand{\bibnamedash}{\leavevmode\raise3pt\hbox to3em{\hrulefill}\space}

\AtEveryBibitem{%
  \clearfield{issn} % Remove issn
  \clearfield{isbn} % Remove isbn
  \clearfield{doi} % Remove doi

  \ifentrytype{online}{}{% Remove url except for @online
    \clearfield{url}
  }
}

\renewbibmacro{in:}{%
    \ifentrytype{article}{}{\printtext{\bibstring{in}\intitlepunct}}}

\DeclareFieldFormat[article,periodical,inreference]{number}{\mkbibparens{#1}}
\DeclareFieldFormat[article,periodical,inreference]{volume}{\mkbibbold{#1}}
\renewbibmacro*{volume+number+eid}{%
    \printfield{volume}%
    \setunit*{\addthinspace}% NEW
    \printfield{number}%
    \setunit{\addcomma\space}%
    \printfield{eid}}

\DeclareFieldFormat[article,inbook,incollection]{title}{\enquote{#1}\addcomma} 

\date{Novembre 2022}
\bbkannee{75\textsuperscript{e} année, 2022--2023}  % Commandes spéciales
\bbknumero{1199}                                      % Commandes spéciales

\title{Pointwise Ergodic Theory: Examples and Entropy}
\subtitle{after Jean Bourgain}

\author{Ben Krause}
\address{King's College London \\
Mathematics Department\\
Strand, London WC2R 2LS}
\email{ben.krause@kcl.ac.uk}

\newcommand{\M}{\mathcal{M}}

\newcommand{\F}{\mathcal{F}}

\theoremstyle{theorem}
\newtheorem{heuristic}[defi]{Heuristic}

\newcommand{\coloneqq}{\mathrel{\mathop:}=}
\newcommand{\eqqcolon}{=\mathrel{\mathop:}}

\begin{document}

\maketitle

\section*{Overview}
Pointwise ergodic theory, the motivation for discrete harmonic analysis, has at its roots the classical theorem of Birkhoff \cite{[2]}, which can be described as follows:

\begin{displayquote}
For every ergodic ---that is, \enquote{sufficiently randomizing}--- measure-preserving transformation, $\tau$, of a probability space, $(X,\mu)$, and any integrable function $f \in L^1(X,\mu)$, $\mu$-almost surely, one can recover the mean of $f$ by considering the Ces\'{a}ro sums
\[ \frac{1}{N} \sum_{n \leq N} f(\tau^nx) \to \int_X f \ d\mu \ \mu-\text{a.e.}\]
\end{displayquote}
Informally, this theorem says that one can recover the \emph{spatial mean} of $f$, 
\[ \int_X f \ d\mu,\]
 by considering the \emph{temporal means} 
\[ \big\{ \frac{1}{N} \sum_{n\leq N} f(\tau^nx) \big\},\]
 formed by \enquote{sampling} the function $f$ at the \enquote{times} $\{\tau^nx\}$ and taking the appropriate average.\footnote{Even in the case when $\tau$ is not ergodic, the temporal means $\big\{ \frac{1}{N} \sum_{n\leq N} \tau^n f(x) \big\}$ still converge $\mu$-almost everywhere.}

A classical question in pointwise ergodic theory concerned the almost-everywhere existence of limiting behavior of averages
\begin{align}\label{H-e:means}
\frac{1}{N} \sum_{n=1}^N \tau^{a_n} f
\end{align}
where $\{a_n\}$ is \enquote{sparse}; as is custom, here and throughout we use $\tau^k f$ to denote the function
\begin{align*}
x \mapsto f(\tau^k x).
\end{align*}
When the lower density of the sequence $\{ a_n\}$ is bounded away from zero
\begin{align*}
\liminf \frac{|\{ n : a_n \leq N\}|}{N} > 0,
\end{align*}
convergence is readily exhibited, and the classical question concerned the existence of sequences $\{ a_n \}$ with zero density,
\begin{align*}
\lim \frac{|\{ n : a_n \leq N\}|}{N} = 0,
\end{align*}
for which the averages \eqref{H-e:means} converged almost everywhere. In \cite{[2]}, such a sequence was constructed; it consisted of taking long blocks of natural numbers, followed by much longer gaps, followed by slightly longer blocks, followed by even longer gaps, etc. In particular, this sequence had an upper Banach density of $1$
\begin{align*}
d^*(\{ a_n \}) \coloneqq  \limsup_{|I| \to \infty \text{ an interval}} \frac{|\{ a_n \} \cap I|}{|I|} = 1.
\end{align*}
The question remained, however, whether or not there existed upper Banach density-zero sequences, $\{ a_n \}$ with $d^*(\{a_n \}) = 0$, for which the almost-everywhere convergence of the averages \eqref{H-e:means} could be proved. In particular, the classical question, explicitly posed first by Furstenberg \cite{[16]}, see also \cite{[1]}, was whether or not the averages along the squares
\begin{align*}
\frac{1}{N} \sum_{n=1}^N \tau^{n^2} f
\end{align*}
converged pointwise almost everywhere, initially for $f \in L^2(X)$. In
breakthrough work, \cite{[5], [6], [9]}, Bourgain answered this question affirmatively, and proved the almost everywhere convergence of \eqref{H-e:means} for any polynomial sequence, 
\[ \{ a_n = P(n) \}, \; \; \; P \in \mathbb{Z}[ \cdot ], \]
and any $f \in L^p(X), \ p >1$, for any $\sigma$-finite measure space $X$; this result was later proven to be sharp \cite{[11], [25]}.

\begin{theo}\label{t:BPoly}
Suppose that $(X,\mu)$ is a $\sigma$-finite measure space, $\tau: X \to X$ is a measure-preserving transformation, and $P \in \mathbb{Z}[\cdot]$ is a polynomial with integer coefficients. Then for each $1 < p < \infty$
\[ \frac{1}{N} \sum_{n=1}^N \tau^{P(n)} f \]
converges $\mu$-a.e.
\end{theo}

Although the issue of pointwise convergence is {qualitative}, Bourgain's insight was to {quantify} the rate at which convergence occurred -- and then to use an abstract {transference} argument first due to Calder\'{o}n \cite{[12]} to deduce these quantitative estimates from a single \enquote{universal} measure preserving system. By considering sequences of the form
\[ \mathbb{Z} \ni n \mapsto \tau^n f(x), \; \; \; x \in X \text{ fixed} \]
and using the measure-preserving nature of $\tau$, Bourgain was able to reduce matters to proving estimates in the case of the integers with counting measure and the shift $(\mathbb{Z}, |\cdot|, \tau: x \mapsto x-1)$.

In particular, Bourgain was after quantitative estimates on the oscillation of the averaging operators
\begin{align}\label{e:foreword-squares1}
 \frac{1}{N} \sum_{n =1}^{N} f(x - P(n)),
 \end{align}
applied first to $\ell^2(\mathbb{Z})$-functions. A natural perspective on \eqref{e:foreword-squares1} is as a convolution of $f$ and
\[ K_N(x) \coloneqq  \frac{1}{N} \sum_{n=1}^N \delta_{P(n)}(x) \]
where $\delta_m$ denotes the point-mass at $m \in \mathbb{Z}$; as this problem is $\ell^2(\mathbb{Z})$-based, the Fourier transform method is naturally employed, and the key to the analysis is an understanding of the exponential sums
\[ \frac{1}{N} \sum_{n \leq N} e^{- 2\pi i \beta \cdot P(n)},\]
which is accomplished via the \emph{circle method} from analytic number theory; the interplay between the \enquote{soft} analytic issue of pointwise convergence and \enquote{hard} analytic estimates on the integers/Euclidean space via analytic-number-theoretic means is characteristic of the fields of pointwise ergodic theory and discrete harmonic analysis.

I first came to understand Bourgain's work by reading \cite{[36]}, which I think explains Theorem \ref{t:BPoly} beautifully; the goal of these notes is to complement \cite{[36]} by trying to explain the motivation behind Bourgain's argument. 

Accordingly, for the sake of clarity, we will shift our focus slightly from proving Theorem \ref{t:BPoly}, and will instead focus on the related maximal estimate, in the representative case of $L^2(X)$.

\begin{theo}\label{t:BMax}
Suppose that $(X,\mu)$ is a $\sigma$-finite measure space, $\tau: X \to X$ is a measure-preserving transformation, and $P \in \mathbb{Z}[\cdot]$ is a polynomial with integer coefficients. Then there exists an absolute constant $\mathbf{C}$, independent of $(X,\mu,\tau)$, so that
\[ \| \sup_N \Big|\frac{1}{N} \sum_{n=1}^N \tau^{P(n)} f \Big| \|_{L^2(X)} \leq \mathbf{C} \cdot \| f\|_{L^2(X)}.\]
\end{theo}

By Calder\'{o}n's transference principle, Theorem \ref{t:BMax} follows from the analoguous estimate of the integers: if we define
\begin{align}\label{e-MAX-1}
\mathscr{M} f(x) \coloneqq  \sup_N \Big| \frac{1}{N} \sum_{n=1}^N f(x- P(n)) \Big|,
\end{align}
then our focus turns to establishing the following estimate
\begin{theo}\label{2-MAINGOAL}
For any $P \in \mathbb{Z}[\cdot]$, the following norm inequality holds: there exists an absolute constant $\mathbf{C}$ so that 
\begin{align*}
\| \mathscr{M}f \|_{\ell^2(\mathbb{Z})} \leq \mathbf{C} \cdot \|f\|_{\ell^2(\mathbb{Z})}.
\end{align*}
\end{theo}

Below, following the lead of \cite{[36]}, we will restrict to the case where 
\[ P(n) = n^d, \]
as this eliminates some number-theoretic technicality while still capturing the essence of the problem.

\subsubsection{Notation}
Here and throughout we abbreviate the complex exponential $e(t) \coloneqq  e^{2 \pi i t}$, so that we may express the Fourier transform in Euclidean space, and on the integers, respectively as
\begin{align*} \hat{f}(\xi) &= \int_{\mathbb{R}} f(x) \cdot e(-\xi x) \ dx, \; \; \; g^{\vee}(x) = \int_{\mathbb{R}} g(\xi)  \cdot e(\xi x) \ d\xi \\
\hat{f}(\beta) &= \sum_n f(n)  \cdot e(-\beta n), \; \; \; g^{\vee}(n) = \int_{\mathbb{T}} g(\beta)  \cdot e(\beta n) \ d\beta. \end{align*}
We will let
\begin{align*} \phi_k(t) \coloneqq  2^{-k} \cdot \phi(2^{-k} \cdot t) \end{align*}
denote the usual $L^1$-normalized dyadic dilations, and for frequencies $\theta$, we let
\begin{align}\label{e-mod} \text{Mod}_\theta g(x) \coloneqq  e(\theta x) \cdot g(x) \end{align}
so that
\[ \widehat{\text{Mod}_\theta g}(\beta) = \hat{g}(\beta - \theta),\]
and recall the Hardy--Littlewood Maximal operator
\begin{align*}
M_{\text{HL}} f(x) \coloneqq  \sup_{r > 0} \, \frac{1}{2r} \int_{-r}^{r} |f(x-t)| \ dt \; \; \; \text{ or } \; \; \; \coloneqq  \sup_{N \geq 0} \, \frac{1}{2N+1} \sum_{n= -N}^{N} |f(x-n)|;
\end{align*}
although we use the same notation to refer to both continuous and discrete maximal operator, it will be clear from context which formulation we use.

We will let $[N] \coloneqq  \{1,\dots,N\}$, and abbreviate $\sum_{n \leq N} \coloneqq  \sum_{n=1}^N$. We will use the symbol $\mathbf{c}$ to denote suitably small constants, which remain bounded away from zero, and $\mathbf{C}$ to denote suitably large constants, which remain bounded above. If we need these constants to depend on parameters, we use subscripts, thus $\mathbf{c}_d$ is a constant that is small depending on $d$. We use $X = O(Y)$ to denote the statement that $|X| \leq \mathbf{C} \cdot Y$, and analogously define $X = O_d(Y).$

Finally, we will use the heuristic notation
\[ f \; \; \; \textrm{``}=\textrm{''} \; \; \; g \]
to denote moral equivalence: up to tolerable errors, $f$ and $g$ exhibit the same type of behavior.

\section{Discrete Complications}

Before beginning our discussion of Theorem \ref{2-MAINGOAL}, let us explain why we might expect this to be a challenging problem.

For problems with a \enquote{linear} flavor, the discrete theory essentially mirrors the continuous theory
\[ \sup_r \frac{1}{r} \int_0^r |f(x-t)| \ dt \; \; \; \textrm{``}=\textrm{''} \; \; \; \sup_N \frac{1}{N} \sum_{n=1}^N |f(x-n)| \]
as can be seen by experimenting with functions of the form $F(\lfloor x \rfloor)$ and using dilation invariance of the real-variable maximal function to reduce attention to real variable functions that are constant on unit scales.

The problems become dramatically more complicated once linearity is destroyed. In this case, we consider the simple example of the Hardy--Littlewood maximal function along the curve $t \mapsto t^d$. The continuous maximal function
\begin{align}\label{e-M2}
M f \coloneqq  M_df &\coloneqq  \sup_r \big| \frac{1}{r} \int_0^r f(x - t^d) \ dt \big| = \sup_r \big| \frac{1}{r} \int_0^{r^d} f(x-t) \ \frac{1}{d t^{1-1/d}} \ dt\big|,
\end{align}
is just a weighted version of $M_{\text{HL}}$ via the pointwise majorization
\begin{multline}\label{e-convmhl}
\frac{1}{r} \int_0^{r^d} |f(x-t)| \frac{1}{d t^{1-1/d}} \ dt \leq \sum_{j = 1}^{\infty} 2^{-j/d} \cdot \big( \frac{2^{j/d}}{r}  \int_{2^{-j} \cdot r^d }^{2^{1-j} \cdot r^d} |f(x-t)| \frac{1}{d t^{1-1/d}} \ dt \big)   \\
\leq \mathbf{C}/d \cdot \sum_{j = 1}^{\infty} 2^{-j/d} \cdot \big( \frac{2^j}{r^d} \int_{2^{-j} \cdot r^d}^{2^{1-j} \cdot r^d} |f(x-t)| \ dt \big) \leq \mathbf{C}/d \cdot \sum_{j=1}^{\infty} 2^{-j/d} \cdot M_{\text{HL}} f(x)\\ \leq \mathbf{C} \cdot M_{\text{HL}} f(x).
\end{multline}
On the other hand, no such trick is available in the study of 
\[ \mathscr{M} f(x) \coloneqq  \mathscr{M}_d f(x) \coloneqq  \sup_N \Big| \frac{1}{N} \sum_{n \leq N} f(x-n^d) \Big|,\]
due to the presence of a smallest scale -- there is no real analogue for an infinitesimal change of variables in the discrete setting.

Passing to the Fourier side actually highlights this difference. We can express both~$M$ and~$\mathscr{M}$ as a maximal operator taken over a lacunary sequence of Fourier multipliers, after exploiting non-negativity. Let us begin with $M$:
\[ Mf(x) \coloneqq  \sup_k \left| \big( V_k(\xi) \hat{f}(\xi) \big)^{\vee}(x) \right|,\]
where
\begin{align}\label{2-V}
V_k(\xi) &\coloneqq  \int_{0}^1 e(-\xi 2^{dk} t^d) \ dt,
\end{align}
so that
\begin{align}\label{2-V0} V_k(\xi) = \int_{0}^1 e( - \xi 2^{dk} t) \frac{1}{d t^{1-1/d}} \ dt  \eqqcolon  \widehat{\mu}( 2^{dk} \xi) =
\begin{cases}
   1 + O \big(2^{dk} |\xi| \big) \\
   O \big( (2^{dk} |\xi|)^{-1/d} \big),
  \end{cases}
  \end{align}
as can be seen by Taylor expanding the exponential around the origin and using the principle of stationary phase (cleverly integrating by parts) for the second estimate. Above, we set
\begin{align}\label{e:mu} \mu(t) \coloneqq  \mu_d(t)  \coloneqq  \frac{1}{d t^{1-1/d}} \cdot \mathbf{1}_{(0,1]}.
\end{align}

What this analysis says is that the multipliers $V_k$ try very hard to look like $\widehat{\varphi_{dk}}$ for, say, a Schwartz function $\varphi \geq 0$ with $\widehat{\varphi}(0)=1$, as in this case, one has similar estimates:
\begin{align}
\widehat{\varphi_{dk}}(\xi) =
\begin{cases}
   1 + O \big(2^{dk} |\xi| \big) \\
   O \big( (2^{dk} |\xi|)^{-100} \big)
  \end{cases}
  \end{align}
(say); compare to \eqref{2-V}. Now, by replacing the weaker $\ell^\infty_k$-norm of $\{ (\mu_{dk} - \varphi_{dk})*f) \}$ with the stronger $\ell^2_k$-norm, we arrive at
\begin{align}\label{2-e:SFNARG}
 Mf &\leq \sup_k |\varphi_{dk}*f| + \sup_k  |(\mu_{dk} - \varphi_{dk} ) * f| \nonumber \\
 & \qquad \leq \mathbf{C} \cdot  M_{\text{HL}} f + \big( \sum_k
| (\mu_{dk} - \varphi_{dk}) * f |^2 \big)^{1/2}  \nonumber \\
& \qquad \qquad \eqqcolon  \mathbf{C} \cdot M_{\text{HL}} f + S f,
\end{align}
where $Sf$ is a so-called \emph{square function}, which is highly-tailored to study $L^2$-based problems. Indeed, we use Plancherel to bound
\begin{align}\label{e-sfxnarg}
\| Sf \|_{L^2(\mathbb{R})}^2 &= \| \big( \sum_k | (\mu_{dk} - \varphi_{dk})*f |^2 \big)^{1/2} \|_{L^2(\mathbb{R})}^2 = \sum_k \| (\mu_{dk} - \varphi_{dk})*f \|_{L^2(\mathbb{R})}^2 \nonumber \\
& \qquad = \sum_k \| (V_k - \widehat{ \varphi_{dk}}) \cdot \hat{f} \|_{L^2(\mathbb{R})}^2 = \int \sum_k |V_k(\xi) - \widehat{\varphi_{dk}}(\xi)|^2 \cdot |\hat{f}(\xi)|^2 \ d\xi \nonumber \\
& \qquad \qquad \leq \sup_\xi \ \sum_k |V_k(\xi) - \widehat{\varphi_{dk}}(\xi)|^2 \cdot \| \hat{f} \|_{L^2(\mathbb{R})}^2 \\
& \qquad \qquad \qquad \leq \mathbf{C} \cdot \sup_\xi \ \sum_k \min \{ 2^{kd} |\xi|,(2^{kd} |\xi|)^{-1/d} \}^2 \cdot \| f \|_{L^2(\mathbb{R})}^2 \nonumber \\
& \qquad \qquad \qquad \qquad \leq \mathbf{C}_d \cdot \| f \|_{L^2(\mathbb{R})}^2, \nonumber
\end{align}
using the fact that $\widehat{\varphi_{dk}}(\xi)$ satisfies the same estimates as $V_k$, namely \eqref{2-V}, so that for $|\xi| \leq \mathbf{C} \cdot 2^{-dk} $
\[ V_k(\xi) - \widehat{\varphi_{dk}}(\xi) = \big( 1 + O(2^{dk}|\xi|) \big) - \big( 1 + O(2^{dk}|\xi|) \big) = O(2^{dk}|\xi|)\]
and when $|\xi| > \mathbf{C} \cdot 2^{-dk}$
\[ V_k(\xi), \ \widehat{\varphi_{dk}}(\xi) = O( (2^{dk} |\xi|)^{-1/d} ).\]

If we try the same trick with the discrete operator $\mathscr{M}$,
\[ \mathscr{M} f(x) = \sup_k |K_k*f(x)|\]
where
\begin{align}\label{2-e:Kk}
K_k(x) \coloneqq  \frac{1}{2^k} \sum_{n \leq 2^k} \delta_{n^d}(x), \end{align}
we can similarly express $\mathscr{M}$ as a maximal multiplier operator
\[ \mathscr{M}f(x) = \sup_{k \geq 0} | \big( \widehat{K_k}(\beta) \hat{f}(\beta) \big)^{\vee}(x) |,\]
where the multipliers $\widehat{K_k}$ are of a different form than the $\{V_k\}$:
\[ \big\{ \widehat{K_k}(\beta) \coloneqq  \frac{1}{2^k} \sum_{m\leq 2^k} e( - \beta m^d)  \big\}_{k\geq 0}.\]
Each multiplier is a \emph{Weyl sum}, and requires the so-called circle method of Hardy and Littlewood to analyze. As we will see below, each multiplier
\[  \widehat{K_k}(\beta) \]
is large and interesting whenever $\beta$ is \enquote{$k$-close} to a rational number with a \enquote{$k$-small} denominator, i.e.\ $\beta$ lives in a so-called \enquote{$k$-major arc}, and is \enquote{$k$-negligible} otherwise, when $\beta$ lives in the complementary \enquote{$k$-minor arc.} In particular, we see subtle \emph{arithmetic} issues that arise as we seek to analyze the relevant multipliers; contrast this to the Euclidean situation, where we were able to understand the multipliers purely according to the \emph{magnitude} of the frequency variable. In other words, whereas the analysis in the Euclidean setting is entirely dictated by the distance from the frequency variable to the distinguished zero-frequency -- multi-frequency issues arise as we seek to understand the multipliers $\widehat{K_k}(\beta)$. Essentially, the main work in bounding
\[ \| \mathscr{M} f\|_{\ell^2(\mathbb{Z})} \leq \mathbf{C} \cdot \|f \|_{\ell^2(\mathbb{Z})},\]
boils down to overcoming these multi-frequency complications.

\section{Examples}
In what follows, we can and will assume that $k$ is sufficiently large depending on $d$.

To come to grips with 
\[  \mathscr{M} f \coloneqq  \sup_k |K_k*f|, \]
we first build some intuition by studying some examples:

Whereas the dilation invariance of the real line allows one to study \eqref{e-M2} or $M_{\text{HL}}$ using examples that live at unit scales, there is no such dilation invariance on $\mathbb{Z}$. Rather, a rough analogue of \enquote{zooming in} is passing to an arithmetic progression. Of course, this analogy is not precise, as arithmetic progressions are characterized by both gap size and diameter. Accordingly, we begin by analyzing the behavior of \eqref{2-e:Kk} when applied to functions
\begin{align}\label{2-ex:APtest}
\varphi_{Q,N} \coloneqq  \mathbf{1}_{Q \mathbb{Z}} \cdot \varphi(\cdot/N)
\end{align}
where $\varphi$ is a smooth bump function, and we think of $Q \leq N^{1/2}$; note the approximation
\begin{align}\label{e:size} \| \varphi_{Q,N} \|_{\ell^2(\mathbb{Z})} \approx (N/Q)^{1/2}.
\end{align}

A common simplifying assumption when passing to arithmetic progressions is that the gap size be prime, as this eliminates various arithmetic technicalities, so we will do so below.

With these reductions in mind, we begin to compute.

\subsection{Example}
%We begin by recording the statistics of $\varphi_{Q,N}$, see \eqref{2-ex:APtest}. While this will not effect the below computation, we will assume that $Q$ is prime.
%\begin{align*}
%\| \varphi_{Q,N} \|_{\ell^1(\mathbb{Z})} \approx \frac{N}{Q}, \; \; \; \| \varphi_{Q,N} \|_{\ell^2(\mathbb{Z})} \approx (\frac{N}{Q})^{1/2}, \; \; \; \| \varphi_{Q,N} \|_{\ell^{\infty}(\mathbb{Z})} \approx 1;
%\end{align*}
%note that 
%\[ \| K_k*\varphi_{Q,N} \|_{\ell^1(\mathbb{Z})} \approx \frac{N}{Q}\]
%by Fubini.
For technical reasons, we will replace the full convolution operator $K_k$, with its smooth \enquote{top half,} in that for a smooth $\mathbf{1}_{[1,2]} \leq \phi \leq \mathbf{1}_{[1/2,4]}$, we consider
\begin{align}\label{2-e:new}
K_k' \coloneqq  \sum_n \phi_k(n) \cdot \delta_{n^d}.
\end{align}
Using convexity, arguing as in \eqref{e-convmhl}, we can bound  
\[ \sup_k |K_k*f| \leq \mathbf{C} \cdot \sup_k |K_k'*f|, \]
so there is no harm in this replacement.

So, we will be interested in understanding
\begin{align}\label{2-e:test} 
K_k'*\varphi_{Q,N}.
\end{align}
There are some scaling considerations that we quickly note:
Since
\begin{align*}
| n^d - (n-1)^d | \geq 2^{k(d-1)} 
\end{align*}
for $2^{k-1} < n \leq 2^{k+2}$, \eqref{2-e:test} becomes trivial if $N \leq  2^{k(d-1)}$, as in this case each element of the sum set
\begin{align*}
\{ n^d : 2^{k-1} < n \leq 2^{k+2} \} + \{ Qj : j \leq N/Q \}
\end{align*}
has $O(1)$ representations of the form $n^d + Qj$. On the other hand since $K_k'$ is supported on $[2^{dk+2}]$, we can assume that $N \leq 2^{dk+2}$, as convolution with $K_k'$ acts independently on intervals separated by $> 2^{dk+2}$. In particular, by translation invariance we can and will restrict to $|x| \leq \mathbf{C} \cdot 2^{dk}$, and assume that 
\begin{align}\label{2-e:ncomp}
N \approx 2^{k(d-1 + \delta)}
\end{align} for some $0 < \delta \leq 1$.

If we use Fourier inversion, we may express
\begin{align}\label{2-e:invconv}
\eqref{2-e:test}  = \int \widehat{K_k'}(\beta) \cdot \widehat{\varphi_{Q,N}}(\beta) \cdot e(\beta x) \ d\beta.
\end{align}
To determine the Fourier transform of $\varphi_{Q,N}$, we express the indicator function of $Q\mathbb{Z}$ as an exponential sum,
\[ \mathbf{1}_{Q\mathbb{Z}}(n) = \frac{1}{Q} \sum_{A=1}^Q e(A/Q \cdot n),\]
and compute
\begin{align}\label{e-psum}
\sum_n \frac{1}{Q} \sum_{A=1}^Q e(A/Q \cdot n) \cdot \varphi(n/N)  \cdot  e(-n \beta) = \frac{1}{Q} \sum_{A=1}^Q N \widehat{\varphi}(N(\beta - A/Q))
\end{align}
by applying Poisson summation to the Schwartz function
\[ t \mapsto \frac{1}{Q} \sum_{A=1}^Q e(A/Q \cdot t) \cdot \varphi(t/N)  \cdot  e(-t \beta) \]
In particular, up to Schwartz-tail considerations, we are only interested in 
\[ \beta \in \mathbb{Z}/Q\mathbb{Z} + O(N^{\mathbf{c}_{d,\delta} -1}),\]
as in the opposite case
\[  \big| \sum_n \frac{1}{Q} \sum_{A=1}^Q e(A/Q \cdot n) \cdot \varphi(n/N)  \cdot  e(-n \beta) \big| \leq \mathbf{C}_{d,\delta} \cdot N^{-100} \]
using the Schwartz decay of $\hat{\varphi}$, see \eqref{e-psum}. So, for such $\beta$, decomposing 
\[ \beta = A/Q + \eta, \; \; \; |\eta| \leq \mathbf{C} \cdot N^{\mathbf{c}_{d,\delta}-1},\]
and $n = pQ + r$, we find that
\begin{align*}
\beta n^d &= (A/Q + \eta) \cdot  (pQ +r)^d \\
& \qquad \equiv A/Q  \cdot r^d + \eta  \cdot (pQ)^d + O(|\eta| \cdot 2^{(d-1)k} \cdot Q) \mod 1,
\end{align*}
so that for such $\beta$
\begin{align}
\widehat{K_k'}(\beta) &= \sum_n \phi_k(n)  \cdot e(-\beta n^d) \nonumber \\
& \qquad = \sum_{pQ +r} \phi_k( pQ + r)  \cdot e(-A/Q  \cdot r^d)  \cdot e(-\eta  \cdot (pQ)^d) + O\big( \frac{2^{k(d-1)} \cdot Q}{N^{1-\mathbf{c}_{d,\delta}}}\big) \nonumber \\
& \qquad \qquad = \frac{1}{Q} \sum_{r=1}^Q e(-A/Q  \cdot r^d) \cdot \sum_{pQ} Q\cdot \phi_k( pQ) \cdot e(-\eta  \cdot (pQ)^d)   + O\big( \frac{2^{k(d-1)} \cdot Q}{N^{1-\mathbf{c}_{d,\delta}}}\big),
\end{align}
using the smoothness of $\phi$. To drop this error terms, we stipulate that $Q \leq 2^{k \delta/2}$, see \eqref{2-e:ncomp}, so that for $|\beta-A/Q| \leq \mathbf{C} \cdot N^{\mathbf{c}_{d,\delta} - 1}$ we may express
\begin{align}
{K_k'}(\beta) = S(A/Q) \cdot \sum_{pQ} Q\cdot  \phi_k( pQ ) \cdot e(-(\beta - A/Q)  \cdot (pQ)^d) +\widehat{ \mathcal{E}_k}(\beta), \nonumber
\end{align}
where $S(A/Q)$ are \emph{complete Weyl sums}, and $\mathcal{E}_k$ is an error term with small Fourier coefficients. Explicitly:
\begin{align*}
S(A/Q) &\coloneqq  \frac{1}{Q} \sum_{n \leq Q} e( - A/Q \cdot n^d ) = \frac{1}{Q} \sum_{m \leq Q} e( - A/Q \cdot m) \cdot |\{ n \leq Q : n^d \equiv m \mod Q \}| 
\end{align*}
precisely captures the equidistribution properties of $n^d \mod Q$, quantified via the upper bound,
\begin{align}\label{2-WEYL} |S(A/Q)| \leq \mathbf{C}_\epsilon \cdot Q^{\epsilon-\frac{1}{d}}, \; \; \; (A,Q) = 1, \;\;\; \epsilon > 0; \end{align}
see \cite{[18]}. And, $\mathcal{E}_k$ is a negligible error term, in that
\[ \| \widehat{\mathcal{E}_k} \|_{L^\infty(\mathbb{T})} \leq \mathbf{C} \cdot 2^{-k\delta/4} \]
(provided $\mathbf{c}_{d,\delta}$ has been chosen appropriately), so that
\begin{align*} 
\| \mathcal{E}_k* \varphi_{Q,N} \|_{\ell^2} = \| \widehat{\mathcal{E}_k} \cdot \widehat{\varphi_{Q,N}} \|_{L^2(\mathbb{T})}  \leq \mathbf{C} \cdot 2^{-k \delta/4} \cdot \| \widehat{ \varphi_{Q,N} } \|_{L^2(\mathbb{T})} 
= \mathbf{C} \cdot 2^{-k \delta/4} \cdot \| { \varphi_{Q,N} } \|_{\ell^2(\mathbb{Z})};
\end{align*}
in what follows, we will discard $\mathcal{E}_k$ from consideration.

By a Riemann summation argument, comparing
\begin{align*} Q \cdot \phi_k(Qp) \cdot e(-(pQ)^d (\beta - A/Q) )   &= \int_p^{p+1} Q \cdot \phi(Qt) \cdot e( - (\beta- A/Q) \cdot (tQ)^d ) \ dt \\
& \qquad + O\big(2^{- \mathbf{c}_{d,\delta} k} \cdot 2^{-k} \cdot Q  \cdot (1 + 2^{-k} \cdot|Qp|)^{-100} \big)
\end{align*}
we approximate, up to pointwise errors of the order $2^{-\mathbf{c}_{d,\delta} k}$
\begin{align*}
&\widehat{K_k'}(\beta) = S(A/Q) \cdot \int \phi(t) \cdot e(-2^{dk}  (\beta - A/Q) \cdot t^d ) \ dt + O(2^{-\mathbf{c}_{d,\delta} k}) \\
%& \qquad = \int \frac{\phi(s^{1/d})}{ds^{1-1/d}} \cdot e(-2^{dk} s(\beta - A/Q)) \ ds \cdot S(A/Q) \\
&  \qquad = S(A/Q) \cdot \int \phi'(s) \cdot e(-2^{dk} (\beta - A/Q) \cdot s) \ ds + O(2^{-\mathbf{c}_{d,\delta} k}) , \; \; \; \; \; \; \phi'(s) \coloneqq  \frac{\phi(s^{1/d})}{ds^{1-1/d}} \\
& \qquad \qquad = S(A/Q) \cdot \widehat{\phi'}(2^{dk}(\beta - A/Q)) + O(2^{-\mathbf{c}_{d,\delta} k})
\end{align*}
where $\phi'$ is Schwartz as well, see \eqref{2-e:new}. Consequently
\begin{align}
\eqref{2-e:invconv} \; \; \; \textrm{``}&=\textrm{''} \; \; \; \frac{1}{Q} \sum_{A \leq Q} S(A/Q) \int N \widehat{\varphi}(N(\beta - A/Q)) \cdot \widehat{\phi'}(2^{dk}(\beta - A/Q)) \cdot e(\beta x) \ d\beta, \nonumber \\
& \qquad = \frac{1}{Q} \sum_{A \leq Q} e(A/Q x) \cdot S(A/Q) \cdot \Phi(x),
\nonumber
\end{align}
where we consolidate
\[ \Phi(x) \coloneqq  \int \varphi((x-2^{dk} s)/N) \cdot \phi'(s) \ ds\]
so that
\[ \hat{\Phi}(\beta) = N \hat{\varphi}(N \beta) \cdot \widehat{\phi'}(2^{dk} \beta),\] 
and thus $\| \Phi \|_{\ell^2(\mathbb{Z})} \approx \frac{N}{2^{dk/2}}$. Summing, we find that
\begin{align}\label{e-comp2}
\| K_k'* \varphi_{Q,N} \|_{\ell^2(\mathbb{Z})}^2  \textrm{``}&=\textrm{''} \sum_x \big| \frac{1}{Q} \sum_{A \leq Q} e(A/Q x) \cdot S(A/Q) \big|^2 \cdot |\Phi(x)|^2 \nonumber  \\
& \qquad  = \frac{1}{Q^2} \sum_{A,B \leq Q} S(A/Q) \cdot  \overline{S(B/Q)} \cdot \sum_x e(( A/Q - B/Q) x) \cdot |\Phi(x)|^2 \nonumber \\
& \qquad \qquad = \frac{1}{Q^2} \sum_{A,B \leq Q} S(A/Q) \cdot \overline{S(B/Q)} \cdot \widehat{ |\Phi|^2}(A/Q - B/Q).
\end{align}
Since 
\[ \widehat{|\Phi|^2} = \hat{\Phi}*\hat{\Phi}^*, \; \; \; \text{ where } \;\; \; g^*(x) \coloneqq  \overline{g(-x)} \]
is essentially supported inside $\{ |\xi| \leq \mathbf{C} \cdot N^{-1} \}$, we have
\begin{align}\label{e-comp} \widehat{ |\Phi|^2}(A/Q - B/Q) = \delta_{A=B} \cdot \| \Phi \|_{\ell^2(\mathbb{Z})}^2 + O( (N/Q)^{-100} )
\end{align}
as whenever $A \neq B$, $|A/Q - B/Q| \geq Q^{-1} \gg N^{-1}$. Substituting \eqref{e-comp} into \eqref{e-comp2}, we find that
\begin{align*}
\| K_k'*\varphi_{Q,N} \|_{\ell^2(\mathbb{Z})}^2 \textrm{``}&=\textrm{''} \frac{1}{Q^2} \sum_{A,B \leq Q} S(A/Q)  \cdot \overline{S(B/Q)} \cdot \delta_{A=B} \cdot \| \Phi \|_{\ell^2(\mathbb{Z})}^2 \\
& \qquad = \frac{1}{Q^2} \sum_{A \leq Q} |S(A/Q)|^2 \cdot \| \Phi \|_{\ell^2(\mathbb{Z})}^2.
\end{align*}
By Hua's estimate \eqref{2-WEYL}, using the fact that $Q$ is prime, we bound
\[ \frac{1}{Q} \sum_{A \leq Q} |S(A/Q)|^2 = \frac{1}{Q} + \frac{1}{Q} \sum_{A \leq Q-1} |S(A/Q)|^2 \leq \mathbf{C}_\epsilon \cdot ( 1/Q + Q^{\epsilon - 2/d})\]
so that we find
\begin{align*}
\| K_k'*\varphi_{Q,N} \|_{\ell^2(\mathbb{Z})} &\leq \mathbf{C}_\epsilon \cdot Q^{\epsilon - 1/d} \cdot Q^{-1/2} \cdot \frac{N}{2^{dk/2}} \\
& \qquad = \mathbf{C}_\epsilon \cdot Q^{\epsilon - 1/d} \cdot (N/2^{dk})^{1/2} \cdot (N/Q)^{1/2} \\
& \qquad \qquad \leq \mathbf{C}_\epsilon \cdot Q^{\epsilon - 1/d} \cdot (N/2^{dk})^{1/2} \cdot \| \varphi_{Q,N} \|_{\ell^2(\mathbb{Z})}
\end{align*}
The prefactor $(N/2^{dk})^{1/2}$ comes from scaling considerations; if we are interested in an estimate that is independent of scale, we arrive at the bound
\[ \| K_k'* \varphi_{Q,N} \|_{\ell^2(\mathbb{Z})} \leq \mathbf{C}_\epsilon \cdot Q^{\epsilon - 1/d} \cdot \| \varphi_{Q,N} \|_{\ell^2(\mathbb{Z})}.\]
In particular, quantitatively, the lower bound
\[ \| K_k'* \varphi_{Q,N} \|_{\ell^2(\mathbb{Z})} \geq \delta \cdot \| \varphi_{Q,N} \|_{\ell^2(\mathbb{Z})} \]
automatically forces a bound on the \enquote{arithmetic complexity} of $\varphi_{Q,N}$ via the estimate
\[ Q \leq \mathbf{C}_\epsilon \cdot \delta^{-d - \epsilon}.\]
In particular, we arrive at the following heuristic:
\begin{heuristic}\label{h-bour}
The only obstruction to 
\[ \| K_k'*f \|_{\ell^2(\mathbb{Z})} \ll \| f \|_{\ell^2(\mathbb{Z})} \]
are \enquote{low arithmetic complexity} considerations. \end{heuristic}

\subsection{The Take-Away}
By an application of Weyl's Lemma, a special case of which is stated below, Bourgain was able to make the previous Heuristic \ref{h-bour} rigorous, concluding that the above range of examples were typical: if we set
\[ \Pi_k(\beta) \coloneqq  \sum_{(A,Q) = 1, Q \leq 2^{\mathbf{c} k}} \widehat{\chi}(2^{(d- \mathbf{c})k}(\beta - A/Q)) \]
for a Schwartz function $\chi$ with
\[ \mathbf{1}_{[-1/4,1/4]} \leq \widehat{\chi} \leq \mathbf{1}_{[-1/2,1/2]}\]
then
\begin{align}\label{e-app} \widehat{K_k}(\beta) = \widehat{K_k}(\beta) \cdot \Pi_k(\beta) + O(2^{-\mathbf{c}' k}), \end{align}
and similarly for $K_k'$. In particular, whenever $\Pi_k(\beta) \neq 1$, then necessarily the conclusion of Weyl's Lemma holds.
\begin{lemm}[Weyl's Lemma, Special Case]\label{2-WEYLL}
Suppose $|\beta - a/q| \leq \frac{1}{q\cdot N^{d - \mathbf{c}}}$ with 
\[ N^{\mathbf{c}} \leq q \leq N^{d - \mathbf{c}}.\] Then there exists some $\mathbf{c}_d > 0$ so that \[  |\frac{1}{N} \sum_{n \leq N} e( -\beta n^d ) | \leq \mathbf{C}_d \cdot N^{-\mathbf{c}_d}.\]
\end{lemm}
At this point, by recycling the reasoning from the previous example, one arrives at the physcial-space approximation
\[ K_k' \; \; \; \textrm{``}=\textrm{''} \; \; \;
L_k' \coloneqq  \sum_{(A,Q) = 1, \ Q \leq 2^{\mathbf{c} k}}  S(A/Q) \cdot \text{Mod}_{A/Q} (\chi_{(d-\mathbf{c})k}*\phi'_{dk}) \]
in that 
\begin{align}\label{e-multclose} \| \widehat{K_k' - L_k'} \|_{L^{\infty}(\mathbb{T})} \leq \mathbf{C}_d \cdot 2^{- \mathbf{c}_d k} \end{align}
%and so the 
%\begin{align*} 
%\| (K_k' - L_k')*f \|_{\ell^2(\mathbb{Z})} &= \| \widehat{K_k' - L_k'} \cdot \hat{f} \|_{L^2(\mathbb{T})} \leq \| \widehat{K_k' - L_k'} \|_{L^\infty(\mathbb{T})} \cdot \| \hat{f} \|_{L^2(\mathbb{T})} \\
%& \qquad \leq \mathbf{C} \cdot 2^{-\mathbf{c}_d k} \cdot \| \hat{f} \|_{L^2(\mathbb{T})} = \mathbf{C} \cdot 2^{-\mathbf{c}_d k} \cdot 
%\| f\|_{\ell^2(\mathbb{Z})} \end{align*}
and so the maximal function is bounded on $\ell^2(\mathbb{Z})$
\begin{align}\label{e-sfxn}
\| \sup_k |(K_k' - L_k')*f| \|_{\ell^2(\mathbb{Z})}^2 \leq  \sup_\beta \ \sum_k |\widehat{K_k'}(\beta) - \widehat{L_k'}(\beta)|^2 \cdot \| f \|_{\ell^2(\mathbb{Z})}^2 \leq \mathbf{C}_d \cdot \| f\|_{\ell^2(\mathbb{Z})}^2,
\end{align}
by arguing as in \eqref{e-sfxnarg}, inserting the quantitative bound \eqref{e-multclose} for the final inequality.

Following Heuristic \ref{h-bour}, it makes sense to decompose $L_k'$ according to the approximate level-sets of the Gauss sums, and seek sufficient decay in $s$ on the $\ell^2(\mathbb{Z})$-norms of maximal functions
\[ \sup_{k \geq \mathbf{C} s} |L_{k,s}'* f|, \]
where 
\[ L_{k,s}'\coloneqq  
\sum_{A/Q \in \mathcal{R}_s}  S(A/Q) \cdot \text{Mod}_{A/Q} (\chi_{(d-\mathbf{c})k}*\phi'_{dk})
\]
for 
\begin{align}\label{e:Rs}
\mathcal{R}_s \coloneqq  \{ (A,Q) = 1, \ 2^{s-1} \leq Q < 2^s \}:
\end{align}
one bounds
\[ \sup_{k} |L_k'*f| = \sup_k |\sum_{s \leq \mathbf{c} k} L_{k,s}'*f| \leq \sum_{s = 1}^{\infty} \sup_{k \geq \mathbf{C} s} |L_{k,s}'*f|.\]
After a little slight of hand, using Plancherel's theorem to morally extract a geometrically decacying prefactor,
\[ L_{k,s}'(x) \; \; \; \textrm{``}=\textrm{''} \; \; \; 2^{- \mathbf{c}_d s} \cdot 
\sum_{A/Q \in \mathcal{R}_s}  \text{Mod}_{A/Q} (\chi_{(d-\mathbf{c})k}*\phi'_{dk})(x)
\]
it suffices to prove the following maximal inequality (possibly for a slightly different choice of $\chi$): 
\[ \| \sup_{k \geq \mathbf{C} s} | \big( \sum_{A/Q \in \mathcal{R}_s}  \text{Mod}_{A/Q} \chi_k \big) * f| \|_{\ell^2(\mathbb{Z})} \leq \mathbf{C}_\epsilon \cdot 2^{\epsilon s} \cdot \| f\|_{\ell^2(\mathbb{Z})}, \; \; \; \epsilon > 0;\]
 by averaging over translations, exploiting the smoothness of $\{ \chi_k : k \geq \mathbf{C} s\}$ at physical scales $2^{\mathbf{C} s}$, it suffices to prove the analogous real-variable inequality:
\[ \| \sup_{k \geq \mathbf{C} s} | \big( \sum_{A/Q \in \mathcal{R}_s}  \text{Mod}_{A/Q} \chi_k \big) * f| \|_{L^2(\mathbb{R})} \leq \mathbf{C}_\epsilon \cdot 2^{\epsilon s} \cdot \| f\|_{L^2(\mathbb{R})}, \; \; \; \epsilon > 0;\]
finally, by exploiting the dilation invariance of $\mathbb{R}$, matters at last reduce to establishing the following multi-frequency maximal estimate, see \cite{[9]}:
\begin{prop}\label{2-MAINPROP}
Suppose that $\Theta \coloneqq  \{ \theta_1,\dots,\theta_N \}$ are $1$-separated,
\[ \text{i.e. } \; \; \; |\theta_i - \theta_j| > 1, \; \; \; i \neq j.\]
Then
\begin{align}\label{e-mm} \| \mathcal{M}_{\Theta} f \|_{L^2(\mathbb{R})} \coloneqq  \| \sup_{k \geq \mathbf{C}} |  \sum_{n \leq N} ( \text{\emph{Mod}}_{\theta_n} \chi_k ) * f| \|_{L^2(\mathbb{R})} \leq \mathbf{C}_\epsilon \cdot N^{\epsilon} \cdot \| f\|_{L^2(\mathbb{R})}, \; \; \; \epsilon > 0.
\end{align}
\end{prop}

The proof of Proposition \ref{2-MAINPROP}, which we will presently establish with a bound on the right side \eqref{e-mm} of the form $\log^2 N$, combines ideas from harmonic analysis, probability theory, and Banach space geometry, and was a creative novelty, having further applications to problems in pointwise ergodic theory \cite{[10],[13],[14]} and to problems in \emph{time frequency analysis}, for instance \cite{[15],[24]}. On the other hand, in some ways, the proof technique was highly constrained: there are only so many ways to control a maximal function on $L^2$, as we will explore below.

\section{The Multi-Frequency Problem}

\subsection{Preliminary Observations}\label{ss-pre}
For what is to follow, we introduce that notation
\begin{align}\label{e-Xi}
\Xi_k f \coloneqq  \sum_{n \leq N} (\text{{Mod}}_{\theta_n} \chi_k )*f,
\end{align}
so that we can express
\[ \M_{\Theta} f = \sup_k |\Xi_k f|; \]
here, as above, $\chi$ is a Schwartz function with
\[ \mathbf{1}_{[-1/4,1/4]} \leq \hat{\chi} \leq \mathbf{1}_{[-1/2,1/2]}\]

While the $\Xi_k$ have oscillatory kernels, they admit a natural projection structure, in that
\[ \Xi_k \Xi_l = \Xi_l, \; \; \; k \geq l + 2, \]
as can be seen by passing to Fourier space, see \eqref{e-fexp0} below; to avoid needless technicality, we will henceforth sparsify our set of scales into parity classes, and restrict our attention to a single class, so that whenever $k > k'$, we necessarily have $k \geq k' + 2$.

As establishing Proposition \ref{2-MAINPROP} is an $L^2$-based problem, to better understand these convolution operators, we pass to Fourier space, and compute
\begin{align}\label{e-fexp0}
\widehat{\Xi_kf}(\xi) = \sum_{n \leq N} \widehat{\chi_k}(\xi - \theta_n) \cdot \hat{f}(\xi)
\end{align}
so that
\[\widehat{\Xi_k}(\xi) = \sum_{n \leq N} \widehat{\chi_k}(\xi - \theta_n),\] 
after conflating the operator with its kernel, so that we can alternatively represent
\begin{align}\label{e-fexp}
 \Xi_k f(x) &= \sum_{n \leq N} e(\theta_n x) \int 
 \widehat{\chi_{k}}(\xi) \hat{f}(\xi + \theta_n) e( \xi x) \ d\xi \nonumber \\
& \qquad = \sum_{n \leq N} e(\theta_n x) \cdot \big( \chi_k*(\text{Mod}_{-\theta_n} f) \big)(x) \nonumber \\
& \qquad \qquad = \sum_{n \leq N} e(\theta_n x) \cdot \big( \chi_k*( \chi* \text{Mod}_{-\theta_n} f) \big)(x) \nonumber \\
& \qquad \qquad \qquad \eqqcolon  \sum_{n \leq N} e(\theta_n x) \cdot (\chi_k*f_{\theta_n})(x),
\end{align}
using the fact that $k \geq \mathbf{C}$ and a brief argument with the Fourier transform to arrive at the reproducing identity 
\[ \chi_k* \chi = \chi.\]
The advantage to passing to the formulation involving $\{ f_{\theta_n} \}$ is that the smoothing effect of convolution with $\chi_k$ has been \enquote{factored} out from the oscillatory exponentials $\{ e(\theta_n x) : n \}$. In particular, heuristically, on intervals of bounded size $\mathbf{C}$, as $k$ gets large only the exponentials should vary: if $|I| = \mathbf{C}$ is an interval of appropriate length, then whenever $x \in I$ and $2^k \gg \mathbf{C}$
\begin{align}\label{2-e:TAYLOR}
I \ni x \mapsto \sum_{ n \leq N} e(\theta_nx) \cdot \phi_k*f_{\theta_n}(x) \; \; \; \textrm{``}=\textrm{''}  \; \; \; \sum_{ n \leq N} e(\theta_nx) \cdot \phi_k*f_{\theta_n}(x_I)
\end{align}
for any $x_I$. In particular, if we subdivide $\mathbb{R}$ into (dyadic) intervals $\{ I : |I| = \mathbf{C} \}$ then on each interval we can estimate
\begin{multline}\label{e-vv}
\int_I |\Xi_k f(x)|^2 \ dx \; \; \; \textrm{``}=\textrm{''} \; \; \; \min_{x_I
  \in I} \, \int_I |\sum_{n} e(\theta_nx) \cdot \chi_k*f_{\theta_n}(x_I)|^2 \
dx \\ \leq \mathbf{C} \cdot \min_{x_I \in I} \, \sum_{n \leq N} |\chi_k*f_{\theta_n}(x_I)|^2 \cdot |I|
\end{multline}
as can be seen by bounding
\begin{align*} 
\| \sum_{n \leq N} e(\theta_n x) \cdot a_n \|_{L^2(I)}^2 &\leq \| \sum_{n \leq N} e(\theta_n x) \cdot a_n \cdot \psi_I(x) \|_{L^2(\mathbb{R})}^2 = \| \sum_{n \leq N} a_n \cdot \widehat{\psi_I}(\xi - \theta_n) \|_{L^2(\mathbb{R})}^2 \\
& \qquad = \sum_{n \leq N} a_n \overline{a_m} \cdot  \langle \widehat{\psi_I}(\cdot - \theta_n), \widehat{\psi_I}(\cdot - \theta_m) \rangle = \sum_{n \leq N} |a_n|^2 \cdot \| \widehat{ \psi_I} \|_{L^2(\mathbb{R})}^2 \\
& \qquad \qquad \leq \mathbf{C} \cdot \sum_{n \leq N} |a_n|^2 \cdot |I|,
\end{align*}
for  $\mathbf{1}_I \leq |\psi_I| \leq \mathbf{C} \cdot (1 + \text{dist}(\cdot,I)/|I|)^{-100}$ with a Fourier transform compactly supported inside $[-1/2,1/2]$; this support constrain ensures that
\[ \psi_I(\xi - \theta_n) \cdot \psi_I(\xi - \theta_m) \equiv 0, \; \; \; n \neq m,\]
and since $|I| \geq \mathbf{C}$, the uncertainty principle is satisfied and such a $\psi_I$ can be chosen.

Seeking uniformity, if we set
\begin{align}\label{e-MAX} \F_{\Theta}(x)^2 \coloneqq  \F_{\Theta,\mathbf{C}}(x)^2 \coloneqq  \sum_{n \leq N} \sup_{k \geq \mathbf{C}} |\chi_k*f_{\theta_n}(x)|^2,\end{align} 
then we have a uniform \emph{norm} bound
\begin{align}\label{e-unifnorm}
\sup_{k \geq \mathbf{C}} \, \| \Xi_k f \|_{L^2(I)} \leq \mathbf{C} \cdot \min_{x_I \in I} \, \F_{\Theta}(x_I) \cdot |I|^{1/2} \leq \| \F_{\Theta} \|_{L^2(I)},
\end{align}
and our task is to control
\begin{align}\label{e-local}
\| \M_{\Theta} f \|_{L^2(\mathbb{R})}  = \big( \sum_{|I| = \mathbf{C}} \| \sup_{k \geq \mathbf{C}} \, |\Xi_k f |\|_{L^2(I)}^2 \big)^{1/2},
\end{align}
where the previous calculation motivates us to split up the real line into intervals of \enquote{small} length and treat the contribution of $\M_{\Theta} f$ on each interval individually. The problem, therefore, boils down to controlling a supremum on $L^2$: we can localize and handle the contribution of each individual $\Xi_k$ via the bound
\begin{align}\label{e-ortmax}
\| \F_{\Theta} \|_{L^2(\mathbb{R})}^2 &= \sum_{n \leq N} \int \sup_{k \geq \mathbf{C}} |\chi_k*f_{\theta_n}(x)|^2 \leq \mathbf{C} \cdot \sum_{n \leq N} \int |f_{\theta_n}(x)|^2 \ dx =
\mathbf{C} \cdot \sum_{n \leq N} \int |\widehat{f_{\theta_n}}(\xi)|^2 \ d\xi 
\nonumber \\
& \qquad = \mathbf{C} \cdot \sum_{n \leq N} \int |\hat{\chi}(\xi)|^2 \cdot |\widehat{f}(\xi + \theta_n)|^2 \ d\xi = \mathbf{C} \cdot \int \sum_{n \leq N} |\hat{\chi}(\xi - \theta_n)|^2 \cdot |\hat{f}(\xi)|^2 \ d\xi \nonumber \\
& \qquad \qquad \leq \mathbf{C} \cdot \sup_{\xi}\ \sum_{n \leq N} |\hat{\chi}(\xi - \theta_n)|^2 \cdot \|\hat{f} \|_{L^2(\mathbb{R})}^2  \leq \mathbf{C} \cdot  \|{f} \|_{L^2(\mathbb{R})}^2, \end{align}
using the separation of the frequencies $|\theta_n - \theta_m| > 1, \ n \neq m$, and the Hardy--Littlewood Maximal function in the first inequality. Our task is to pass from \emph{uniform} control of the $\{ \Xi_k : k \}$ to \emph{simultaneous} control, via $\mathcal{M}_{\Theta}$. This is a task that arises frequently -- but is often constrained, as we pause to explore. 

\subsection{Bounding a Supremum on $L^2$}
Suppose that $\{ F_k  \} \in L^2(X)$ is a collection of functions on a measure space, and we are interested in controling
\begin{align}\label{2-e:L2MAX}
\| F_* \|_{L^2(X)} \coloneqq  \| \sup_k |F_k | \|_{L^2(X)}.
\end{align}
To the best of my knowledge, there are essentially four ways to control $F_*$ on $L^2$:
\begin{itemize}
\item Martingale/stopping time methods, like those used to prove Doob's Maximal Inequality from martingale theory, or the closely linked Hardy--Littlewood Maximal Inequality;
\item Semigroup methods, like those used in the Hopf--Dunford--Schwartz Maximal Theorem, a special case of which implies dimension independent bounds on the maximal function $\sup_t |e^{t \triangle} f|$;
\item $TT^*$ orthogonality methods, in which the supremum $F_*$ is realized as a particular linear operator,
\begin{align*}
T\{ F_k\} \coloneqq  \sum_t \mathbf{1}_{E_k} F_k
\end{align*}
for $\{ E_k \}$ a disjoint partition of $X$, and then $T$ is composed with its adjoint, to efficiently compute
\begin{align*}
\| T\|_{L^2(X) \to L^2(X)} = \| T T^* \|_{L^2(X) \to L^2(X)}^{1/2};
\end{align*}
this technique is common in oscillatory integral situations; and
\item Entropy arguments, which leverage vestigial smoothness in the map $k \mapsto F_k(x)$ to control $F_*$.
\end{itemize}

Of the four methods, the oscillatory nature of the averages $\{ \Xi_k : k \}$ precludes a direct argument involving the first method, which gives a privileged role to the zero frequency (expectation);
the serious failure of the identity
\[ \Xi_k \Xi_l \neq \Xi_{k+l} \]
precludes the second method.

As for the $TT^*$ approach, if we linearize our supremum and consider the operator
\begin{align*}
T f(x) = \sum_k \mathbf{1}_{E_k}(x) \cdot \int \sum_{n \leq N} e(\theta_n x) \int \chi_k(x-y) e(-\theta_n y) f(y)  \ dy,
\end{align*}
then the dual operator, $T^*$, is given by
\begin{align*}
T^*g(x) = \sum_r \sum_{n \leq N} e(\theta_n y) \cdot \int e(- \theta_n x) \cdot (g \cdot \mathbf{1}_{E_r})(x) \overline{\chi_k}(x-y) \ dx,
\end{align*}
and nothing is really gained by composition.

Accordingly, we turn our attention to the entropic approach to bounding a supremum on $L^2$.

\section{From Bourgain's Toolkit: The Entropic Method}
This section reviews material over which Bourgain had total command at the time of \cite{[9]}; see \cite{[4],[5],[7],[8]} or even $\S 3$ of \cite[\S 3]{[9]} for representative examples, and
$\S 6$ of \cite[\S 6]{[35]} for an excellent summary. In particular, I imagine that the information Bourgain gleaned from the above Subsection~\S \ref{ss-pre} was enough to guide him directly to the below Section~\S \ref{s-shape}. While the implementation of this approach in studying $\M_{\Theta}$ seems magical upon first reading \cite{[2]}, or in my case the exposition of \cite{[36]}, my hope is that after fully digesting the following material, the reader is able to understand the intution behind the way Bourgain came to his argument.

The basic mechanism behind the entropic approach is to leverage \enquote{size} and \enquote{smoothness,} or rather \enquote{stickiness,} in the parameter space to control a supremum. In terms of our problem at hand, we have uniform control over each average $\Xi_k$ via \eqref{e-unifnorm}, and we search for some notion of smoothness/stickiness to complement this uniformity.

To show off this interplay, we review the following example. 

\begin{lemm}[Sobolev Embedding Lemma]\label{2-l:SOBEMB}
Suppose that $I$ is an interval, and that $F(x,\cdot)$ is absolutely continuous for almost every $x$ with an $L^2$ density. Then the following pointwise estimate holds:
\begin{align}
F_I(x) \coloneqq  \sup_{t \in I} |F(x,t)| \leq \mathbf{C} \cdot |F(x,t_I)| + \mathbf{C} \cdot \left( \int_I |F(x,t)|^2 \ dt \right)^{1/4} \cdot \left( \int_I |\partial_t F(x,t)|^2 \ dt \right)^{1/4} \nonumber 
\end{align}
for any $t_I \in I$. In particular, if
\begin{align}\label{e-multbound}
\sup_{t\in I} \| F(x,t) \|_{L^2_x} \leq A \; \; \; \text{ and } \; \; \; 
\sup_{t \in I} \| \partial_t F(x,t) \|_{L^2_x} \leq a   
\end{align}
then
\begin{align}
\| \sup_{t \in I} |F(x,t)| \|_{L^2_x} \leq \mathbf{C} \cdot \big( A +  (A a |I|)^{1/2}  \big) \nonumber 
\end{align}
\end{lemm}
\begin{proof}
For any $t \in I$, we may bound
\begin{align}
F(x,t)^2 = F(x,t_I)^2 + \int_{[t_I,t]} \partial_s \big( F(x,s)^2 \big) \ ds, \nonumber 
\end{align}
so
\begin{align}\label{e-sq}
|F(x,t)|^2 &\leq |F(x,t_I)|^2 + 2 \int_{I} |F(x,t)| \cdot |\partial_t F(x,t)| \ dt   \nonumber \\
& \qquad \leq |F(x,t_I)|^2 + 2 \left( \int_I |F(x,t)|^2 \ dt\right)^{1/2} \cdot \left( \int_I |\partial_t F(x,t)|^2 \ dt\right)^{1/2} 
\end{align}
The right-hand side of \eqref{e-sq} is independent of $t$, so we can take the supremum in $t$ over the left-hand side of \eqref{e-sq} and then integrate in $x$, applying Cauchy--Schwarz to handle the $L^2$-based $t$-averages.
\end{proof}

While Lemma \ref{2-l:SOBEMB} is very cheap, it is surprisingly robust, and is very useful in studying maximal multiplier operators of the form 
\[ \sup_t |( \hat{f} \cdot \mathbf{m}(t \cdot) )^{\vee}|\]
for bounded $m \in \mathcal{C}^1(\mathbb{R} \smallsetminus \{0\})$, see Lemma $3$ of \textcite[Lemma 3]{[4]}.

It is helpful to discretize this argument: for each $v \geq 1$, define
\begin{align*}
\Lambda_v \coloneqq  \big( 2^{-v} \cdot \mathbb{Z} \big) \cap I \end{align*}
and define the parent of $t \in \Lambda_v$, $\rho(t)\in \Lambda_{v-1}$ to be the minimal element so that 
\begin{align*}
B(t,2^{-v}) \cap B(\rho(t),2^{1-v}) \neq \emptyset, \; \; \; \; \; \; B(x,s) \coloneqq  \{ y : |x-y| < s\}.
\end{align*}

Given $x$-a.e. continuity in $t \mapsto F(x,t)$, to study $F_I$, it suffices to bound
\begin{align*}
\sup_{t \ \in \ \bigcup_{v \geq 1} \Lambda_v} |F(x,t)|;
\end{align*}
by monotone convergence, it suffices to estimate, uniformly in finite subsets $T \subset \bigcup_{v \geq 1} \Lambda_v$,
\begin{align*}
F_T(x) \coloneqq  \sup_{t \in T  } |F(x,t)|.
\end{align*}

To do so, for each $t \in T$, we may telescope
\begin{align*}
t = (t - \rho(t)) + (\rho(t) - \rho^2(t)) + \dots + \rho_0(t)
\end{align*}
where $\rho^j$ is the $j$th composition of $\rho$, and $\rho_0(t)$ is the appropriate composition so that $\rho_0(t) \in \Lambda_{v_0}$ for some $v_0$ to be determined below.

Note that the number of increments required to arrive at a representative $\rho_0 \in \Lambda_{v_0}$ is uniformly bounded, since $T$ is finite. We bound
\begin{align*}
F_T(x) &\leq \sup_{t \in \Lambda_{v_0}} |F(x,t)| + \sum_{v > v_0} \sup_{t \in \Lambda_{v}} | F(x,t) - F(x,\rho(t))| \\
& \qquad \leq \Big( \sum_{t \in \Lambda_{v_0}} |F(x,t)|^2 \Big)^{1/2} + \sum_{v > v_0} \Big( \sum_{t \in \Lambda_{v} } | F(x,t) - F(x,\varrho(t))|^2 \Big)^{1/2},
\end{align*}
noting that all sums are in fact finite, 
and take $L^2_x$-norms, before optimizing over $v_0 \geq 0$ to derive the desired upper bound:
\[ A \cdot |\Lambda_{v_0}|^{1/2} + a \cdot \sum_{v > v_0} |\Lambda_v|^{1/2} \cdot 2^{-v} \leq \mathbf{C} \cdot \big(  A + (A a |I| )^{1/2} \big);\]
see \eqref{e-multbound}.

In both of these arguments, we relied upon smoothness in the map $t \mapsto F(x,t)$. Really, though, we were relying on decaying contributions from 
\begin{align}\label{e-sm0}
\Lambda_v \ni t \mapsto |F(x,t) - F(x,\rho(t))|
\end{align}
as $v$ grows, and the controlled \emph{entropy estimate}
\[ |\Lambda_v| \leq  \mathbf{C} \cdot 2^v \cdot |I|; \]
from the metric perspective, this estimate is measuring the extent to which elements in $I$ adhere to each other ---\enquote{stick together}--- at scales $2^{-v}$. Estimates like
\[ \sup_{t \in I} \| \partial_t F(x,t) \|_{L^2_x} \leq a \]
allow us to capture the smallness in \eqref{e-sm0} in an $L^2$-average sense. But, we may also pointwise approximate $\{ F(x,t) : t \in I \}$
 more directly using a similar telescoping mechanism.

For $T$ as above, consider the set 
\begin{align}
X(x) \coloneqq  X_T(x) \coloneqq  \big\{ F(x,t) : t \in T \big\},
\end{align}
and for each $v$ so that $2^{-v} \leq 2 \cdot \text{diam}(X(x)) $, define $\Lambda_v(x) \subset T$ to be a collection of times $t$ so that
\begin{align}
X(x) \subset \bigcup_{t \in \Lambda_v(x)} B\big( F(x,t), 2^{-v} \big)
\end{align}
subject to the constraint that $|\Lambda_v(x)|$ is minimal; the cardinality is essentially the $2^{-v}$-\emph{entropy} of the set.

Now, let $V$ be so large that each element of $T$ is separated by $> 2^{1-V}$, so that $T = \Lambda_V(x)$. And define the parent of $t \in \Lambda_v(x)$, $\varrho(t)\in \Lambda_{v-1}(x)$ to be the minimal element so that 
\begin{align}\label{e-int}
B\big( F(x,t), 2^{-v} \big) \cap B \big( F(x,\varrho(t)),2^{1-v} \big) \neq \emptyset.
\end{align}
For any $s \in T$, we may similarly bound 
\begin{align}\label{2-e:entexp}
F_T(x) &\leq |F(x,s)|  + \sum_{v} \sup_{t \in \Lambda_v(x)} |F(x,t) - F(x,\varrho(t))|  \nonumber \\
& \qquad \leq |F(x,s)|  + \sum_{v} \Big( \sum_{t \in \Lambda_v(x)} |F(x,t) - F(x,\varrho(t))|^2 \Big)^{1/2}
  \\
& \qquad \qquad \leq |F(x,s)| + \mathbf{C} \cdot \sum_v 2^{-v} \cdot |\Lambda_v(x)|^{1/2}, \nonumber 
\end{align}
as we may bound
\[ |F(x,t) - F(x,\rho'(t))| < 2^{-v} + 2^{1-v} < 2^{2-v} \]
for each $t \in \Lambda_v(x)$ by \eqref{e-int}.
It is convenient to change perspectives and bound 
\begin{align} \label{e-ent}
|\Lambda_v(x)| \leq N_{2^{-v}}(x)
\end{align}
where
\begin{multline*}%\label{e-greedy} 
N_\lambda(x) \coloneqq  \sup \Bigl\{ K : \text{ there exists a sequence of
  times }\Bigr. \\ \Bigl.  t_0 < t_1 < \dots < t_K :  |F(x,t_i) - F(x,t_{i-1})| > \lambda \Bigr\}  \end{multline*}
is a so-called (greedy) \emph{jump-counting function} at altitude $\lambda > 0$, which measures the extent to which $\{ F(x,t) : t \}$ \enquote{stick together} at the scale $\lambda$:
\begin{align*}
N_\lambda(x) < \infty \text{ for all $\lambda > 0$} &\iff \{ F(x,t) : t \} \text{ converges} \\
& \qquad \iff \{ F(x,t) : t \} \text{ \enquote{stick together} at all scales}. \end{align*}
To establish \eqref{e-ent}, one majorizes the left hand side and minorizes the right hand by the $2^{1-v}$-\emph{entropy} of the set: the size of the largest set of $2^{1-v}$-separated points inside of $\{ F(x,t) : t \in I \}$.

The reverse bound 
\[ N_{2^{-v}}(x) \leq |\Lambda_{v+1}(x)|  \]
is simpler, so there is nothing lost quantitatively from this change, as indeed
\[ \sum_v 2^{-v} \cdot |\Lambda_v(x)|^{1/2} \leq \sum_v 2^{-v} \cdot N_{2^{-v}}(x)^{1/2} \leq \mathbf{C} \cdot \sum_v 2^{-v} \cdot |\Lambda_v(x)|^{1/2}. \]

In many special examples, one is able to prove a uniform bound
\begin{align}\label{e-unif} \sup_v \| 2^{-v} \cdot N_{2^{-v}}^{1/2} \|_{L^2} \leq \mathbf{C} \cdot A, \end{align}
which says that in an $L^2$-averaged sense
\begin{align*}
N_{2^{-v}} \; \; \; \textrm{``}\leq \textrm{''} \; \; \; \mathbf{C} \cdot A^2 \cdot 2^{2v} 
\end{align*}
i.e.\ that it costs a quadratically growing price to cover the collection of data $\{ F(x,t) : t \}$ by balls of a given radius. The following examples are representative.

\subsubsection*{Entropic Example One}
Consider the (discrete-time) averaging operators,
\begin{align}\label{e:cond} F(x,t) = \mathbb{E}_k f(x) \cdot \mathbf{1}_{[2^k,2^{k+1})}(t)
\end{align}
where 
\begin{align}
\mathbb{E}_k f (x) \coloneqq  \sum_{|I| = 2^k \text{ dyadic}} \big( \frac{1}{|I|} \int_I f(t) \ dt \big) \cdot \mathbf{1}_I(x)
\end{align}
is the conditional expectation operator, 
projecting onto the $\sigma$-algebra generated by the dyadic intervals $\{ 2^k \cdot [n,n+1) : n \in \mathbb{Z} \}$. The \enquote{stopping-time} structure embedded in the definition of $N_{2^{-v}}$ allows one to neatly employ methods from dyadic harmonic analysis -- secretly, martingale techniques -- to establish \eqref{e-unif}.
\subsubsection*{Entropic Example Two}
To the extent that 
\[ \mathbb{E}_k f \; \; \; \textrm{``}=\textrm{''} \; \; \; \chi_k*f, \]
in that both operators \enquote{blur} at spatial scales $2^{k}$, discarding \enquote{fine scale} information below this threshold, and preserving \enquote{coarse scale} properties that can be detected above this spatial threshold, one can combine a {square function argument} with further orthogonality arguments, in particular the quantitative bound
\begin{align}\label{e-orth} \| \mathbb{E}_k \psi_l - \chi_k*\psi_l \|_{L^2(\mathbb{R})} \leq \mathbf{C} \cdot 2^{- |k-l|/2} \cdot \| \psi_l \|_{L^2(\mathbb{R})}, \; \;\; \psi_l \coloneqq  \chi_l - \chi_{l-1},\end{align}
to extend \eqref{e-unif} to the case where
\begin{align}\label{e-condavg} F(x,t) = 
f*\chi_k(x) \cdot \mathbf{1}_{[2^k,2^{k+1})}(t), \end{align}
and similarly with $\chi$ replaced with any other Schwartz function with $\hat{\chi}(0)  =1$. These ideas first appeared in \cite{[20]}.

\subsection{The Jump-Counting Approach to Entropy}
While the uniform estimate \eqref{e-unif} is a priori insufficient to control the full supremum over $t \in T$, this entropic argument yields a remarkable strengthening over the trivial estimate
\[ \| F_T \|_{L^2_x} \leq \| S_T \|_{L^2_x} \leq |T|^{1/2} \cdot A,\]
where we set
\[ S_T(x)^2 \coloneqq  \sum_{t \in T} |F(x,t)|^2. \]
In particular, for any $t \in T$,
\begin{align}\label{2-e:jumps}
F_T(x) &\leq |F(x,t)| + \mathbf{C} \cdot \sum_v 2^{-v}  \cdot N_{2^{-v}}(x)^{1/2} 
\nonumber \\
& \qquad \leq |F(x,t)| + \frac{S_T(x)}{|T|^{1/2}} + \sum_{v: \frac{S(x)}{|T|^{1/2}} \leq 2^{-v} \leq 2 \cdot S(x) } 2^{-v}  \cdot N_{2^{-v}}(x)^{1/2}
\end{align}
so that, essentialy, the uniform bound \eqref{e-unif} implies\footnote{There
  is a natural comparison between this estimate and the abstract Hilbert space \emph{Rademacher--Menshov} inequality, which also states that 
\begin{align}\label{e-logloss} \| \sup_{n \leq N} |F(x,n)| \|_{L^2} \leq \mathbf{C} \cdot \log N \cdot A 
\end{align}
 under orthogonality constrains on the functions $\{  F(\cdot,n) : n \}$. The analogy is at the level of proof and is that of Lebesgue integration to Riemann integration: the entropy bound organizes the data $ \{ F(x,n) : n \} $ according to its image, while the Rademacher--Menshov inequality is proven by analogously organizing the data according to the domain of the time parameter $n \in [N]$.}
\begin{align}\label{e-genRM} \| F_T \|_{L^2_x} \; \; \; \textrm{``}\leq\textrm{''} \; \; \; \mathbf{C} \cdot \log |T| \cdot A.
\end{align}
%By considering the sequence of times $\{ v_ k \}$ so that 
%\[ v_k \coloneqq  \min\{ v : N_{2^{-t}} \geq 2^{k-1}\} \]
%and bounding
%\[ N_{2^{-v_k}} \leq N_{2^{-v}} < 2 N_{2^{-v_k}} \]
%for $v_k \leq v < v_{k+1}$, one can make this precise by summing
%\begin{align*}
%&|F_T(x)| \leq |F(x,t)| + \mathbf{C} \cdot \sum_v  2^{-v} \cdot N_{2^{-v}}(x)^{1/2} \leq |F(x,t)| + \mathbf{C} \cdot \sum_{2^k \leq 2N} \sum_{[v_k,v_{k+1})} 2^{-v}  \cdot N_{2^{-v}}(x)^{1/2} \\
%& \qquad \leq |F(x,t)| +\mathbf{C} \cdot \sum_{2^k \leq 2N} \sum_{[v_k,v_{k+1})} 2^{-v}  \cdot N_{2^{-v_k}}(x)^{1/2}  \leq |F(x,t)| + \mathbf{C} \cdot \sum_{2^k \leq 2N} 2^{-v_k}  \cdot N_{2^{-v_k}}(x)^{1/2},
%\end{align*}
%leading to \eqref{e-genRM}.
In point of fact, as we will see below, \eqref{e-genRM} often holds with a $\log^2 |T|$ prefactor.
\subsection{Introduction to Variation}
As the difficulty with the heuristic justification for \eqref{e-logloss} shows, see \eqref{2-e:jumps}, a major problem is that, in general, we cannot expect a uniform bound on 
\[ x \mapsto \sup_v \ 2^{-v} \cdot N_{2^{-v}}(x)^{1/2},\]
see \cite{[22]} or \cite{[33]}.

To get around this issue, one instead sacrifices the power $1/2 \to 1/r, \ r > 2$ and introduces the so-called \emph{$r$-variation} of $\{ F(x,t) : t \in I\}$
\begin{align}
\mathcal{V}^r(x) \coloneqq  \mathcal{V}^r_F(x) \coloneqq  \sup \big( \sum_i |F(x,t_i) - F(x,t_{i-1})|^r \big)^{1/r},
\end{align}
where the supremum runs over all finite increasing subsequences inside of $I$. Unlike the jump counting function, the $r$-variation operators crucially satisfies a triangle inequality,
\[ \mathcal{V}^r_{F+G} \leq \mathcal{V}^r_F + \mathcal{V}^r_G, \]
and one may bound
\[ \sup_v \ 2^{-v} \cdot  N_{2^{-v}}(x)^{1/r} \leq \mathcal{V}^r(x), \]
which is important, as the $\mathcal{V}^r$ operators often admit a strong $L^2$-theory. In particular, if $|T| = N$, so that $N_\lambda \leq N$ for all $\lambda$, we may bound
\begin{equation}\label{e-interptrick}
2^{-v} \cdot N_{2^{-v}}^{1/2} \leq 2^{-v}  \cdot N_{2^{-v}}^{1/r} \cdot N_{2^{-v}}^{1/2 - 1/r} \leq N^{1/2 - 1/r} \cdot  \mathcal{V}^r 
\end{equation}
and if we set $r = 2 + \frac{\mathbf{c}}{\log N}$, then we eliminate the pre-factor of $N^{1/2-1/r}$ and end up with the bound
\begin{align*}
2^{-v} \cdot N_{2^{-v}}^{1/2} \leq \mathbf{C} \cdot \mathcal{V}^r, \; \; \;  \; \; \; r = 2 + \frac{\mathbf{c}}{\log N}.
\end{align*}
Substituting into \eqref{2-e:jumps}, we bound, for any $t \in T$
\[ F_T(x) \leq |F(x,t)| + \frac{S_T(x)}{N^{1/2}} + \sum_{v : \frac{S_T(x)}{N^{1/2}} \leq 2^{-v} \leq S_T(x)} \mathcal{V}^r(x) \leq |F(x,t)| + \frac{S_T(x)}{N^{1/2}} + \log N \cdot \mathcal{V}^r(x),\]
which says that
\[ \| F_T \|_{L^2} \leq \mathbf{C} \cdot \big( A + \log N \cdot \| \mathcal{V}^r \|_{L^2} \big), \; \; \; \; \; \; r = 2 + \frac{\mathbf{c}}{\log N}, \]
so control over the $r$-variation operators leads, essentially, to \eqref{e-genRM}.

The relevant estimates for $\mathcal{V}^r$ derive, in many cases, from the following inequality, classically used as a convergence result in martingale theory \cite{[26]}; see \cite{[21]} for a discussion, and \cite{[17]} or \cite{[32]} for more exotic examples.
\begin{prop} [L\'{e}pingle's Inequality, Special Case] \label{p-LEP}
The following estimate holds in the conditional expectation case \eqref{e:cond}:
\[ \| \mathcal{V}^r \|_{L^2(\mathbb{R})} \leq \mathbf{C} \cdot \frac{r}{r-2} \cdot A. \]
\end{prop}
Proposition \ref{p-LEP} extends similarly to the case of convolution operators \eqref{e-condavg}: by combining a square function argument
\begin{align}\label{e-trivar} \mathcal{V}^r( \chi_k*f: k) &\leq \mathcal{V}^r( \mathbb{E}_k f : k) + \mathcal{V}^r( \chi_k*f - \mathbb{E}_k f:k) \nonumber \\ & \qquad \leq \mathcal{V}^r( \mathbb{E}_k f:k) + 2 \cdot \big( \sum_k |\chi_k*f - \mathbb{E}_k f|^2 \big)^{1/2} \end{align}
with the estimates \eqref{e-orth} introduced above, one can use orthogonality techniques and Proposition \ref{p-LEP} to bound both terms in \eqref{e-trivar}. Above, we define the discrete-time variation 
\[ \mathcal{V}^r(f_k : k)(x) \coloneqq  \sup \big( \sum_i |f_{k_i}(x) - f_{k_{i-1}}(x)|^r \big)^{1/r} \]
where the supremum runs over all finite subsequences $\{ k_i \}$.

While the $\mathcal{V}^r$ operators are more delicate than the pertaining maximal functions,
\[ F_T(x) \leq |F(x,t)| + \mathcal{V}^r(x) \]
for any $t \in T$, they are essentially of even strength, in that we have the following heuristic:
\begin{heuristic}\label{heur}
In either case \eqref{e:cond} or \eqref{e-condavg}, it is very hard for $\mathcal{V}^r$ to be large when both $F_I$ and the square function
\[ S_I(x) \coloneqq  \big( \sum_{k: 2^k \in I} |F(x,2^k) - F(x,2^{k+1})|^2 \big)^{1/2} \]
are small: $
\mathcal{V}^r \; \; \; \textrm{``}\approx\textrm{''} \; \; \; \frac{r}{r-2} \cdot (F_I + S_I) \; \; \; \textrm{``}\approx\textrm{''} \; \; \;  \frac{r}{r-2} \cdot F_I,  \ \frac{r}{r-2} \cdot S_I$.\footnote{The close link between maximal function and square function in either context \eqref{e:cond} or \eqref{e-condavg} is classical; see e.g.\ \textcite[\S 1]{[34]}. The relationship between $\mathcal{V}^r, F_I, S_I$ in the case \eqref{e:cond} is via the following good-$\lambda$ inequality:
\[ |\{ \mathcal{V}^r > \mathbf{C} \lambda, F_I,S_I \leq \gamma \lambda \}| \leq \mathbf{C} \cdot \big( \frac{r}{r-2} \big)^2 \cdot \gamma^2 \cdot |\{ \mathcal{V}^r > \lambda \} |, \; \;\; r > 2;\]
see \textcite[\S 3]{[23]}.
}
\end{heuristic}
Finally, and significantly, given our vector-valued perspective on studying
\[ \{ f_{\theta_1},\dots,f_{\theta_N} \}, \]
see \eqref{e-vv}, we observe that just as do the maximal function and square function, the $\mathcal{V}^r$ operators interact well in the vector-valued setting: for sequence-space valued functions $\vec{F} = (F_1,F_2,\dots)$
\begin{align}\label{e-vrvec}
\mathcal{V}^r_{\vec{F}}(x) \coloneqq  \sup \big( \sum_i \| F_n(x,t_i) - F_n(x,t_{i-1}) \|_{\ell^2_n}^r \big)^{1/r} \leq \| \mathcal{V}^r_{F_n}(x) \|_{\ell^2_n},
\end{align}
by Minkowski's inequality for sequence spaces (as $r > 2$), where the supremum runs over finite increasing subsequences of $\{t_i\}$.\\

With this section in mind, we begin to see how Bourgain developed his argument.

\section{The Argument Takes Shape}\label{s-shape}
Bourgain's task was to establish \eqref{e-local}, where we are only thinking about the case where $2^k$ is very large relative to $|I| = \mathbf{C}$. By monotone convergence, we can restrict to finitely many scales $k \in T \subset \mathbb{N} \cap [\mathbf{C},\infty)$. We focus on the case of a single interval.

Guided by our heuristic analysis, we let $x_I \in I$ be a point to be determined later, and seek to bound
\[ \| \sup_k | \sum_{n\leq N} e(\theta_n x) \cdot \chi_k*f_{\theta_n}(x_I)| \|_{L^2_x(I)}.
\]
As discussed above -- we are essentially forced to use the entropic approach. Specifically, we set
\begin{align}\label{e-ftheta} X(x_I) \coloneqq  \{ \chi_k*\vec{f_{\Theta}}(x_I) \coloneqq  \big( \chi_k*f_{\theta_1}(x_I),\dots,\chi_k*f_{\theta_N}(x_I) \big) : k \} \end{align}
and let $\vec{N}_{\lambda}$ denote the appropriate jump-counting function at altitude $\lambda$ with respect to the sequence space norm, $\ell^2([N])$, 
\begin{align*}%\label{e-greedy} 
\vec{N}_\lambda(x) \coloneqq  \sup \Big\{ K &: \text{ there exists a sequence of times }  \mathbf{C} \leq k_0 < k_1 < \dots < k_K :\\
& \qquad 
\| \chi_{k_i}*f_{\theta_n}(x) - \chi_{k_{i-1}}*f_{\theta_n}(x) \|_{\ell^2_{n \in [N]}} > \lambda \Big \}.
\end{align*}
By arguing as above we can bound 
\begin{multline}\label{e-L2close}
\| \sup_k | \sum_{n\leq N} e(\theta_n x) \cdot \chi_k*f_{\theta_n}(x_I)| \|_{L^2_x(I)} \\  \leq 
| \Xi_{k_0} *f(x_I)| \cdot |I|^{1/2}  +
\sum_v \| \max_{k \in \Lambda_v(x_I)} |\sum_{n \leq N} e(\theta_n x) \cdot \big( \chi_k - \chi_{\varrho(k)} \big) *f_{\theta_n}(x_I)| \|_{L^2_x(I)}
\end{multline}
for any $k_0 \geq \mathbf{C}$, see \eqref{e-Xi}. The first term is of a simpler nature, so we will temporarily suppress it; and for each individual $v$ we may bound
\begin{align}\label{e-L2bound}
&\| \max_{k \in \Lambda_v(x_I)} |\sum_{n \leq N} e(\theta_n x) \cdot \big( \chi_k - \chi_{\varrho(k)} \big) *f_{\theta_n}(x_I)| \|_{L^2_x(I)} \nonumber \\
& \qquad 
 \leq \| \big( \sum_{k \in \Lambda_v(x_I)} |\sum_{n\leq N} e(\theta_n x) \cdot \big( \chi_k - \chi_{\varrho(k)} \big) *f_{\theta_n}(x_I)|^2 \big)^{1/2} \|_{L^2_x(I)} \nonumber \\
& \qquad \qquad \leq \big( \sum_{k \in \Lambda_v(x_I)} \| \sum_{n\leq N} e(\theta_n x) \cdot \big( \chi_k - \chi_{\varrho(k)} \big) *f_{\theta_n}(x_I) \|_{L^2_x(I)}^2 \big)^{1/2} \nonumber \\
& \qquad \qquad \qquad \leq \mathbf{C} \cdot 2^{-v} \cdot \vec{N}_{2^{-v}}(x)^{1/2} \cdot |I|^{1/2} 
\end{align}
by arguing as in \eqref{2-e:entexp}, applying \eqref{e-vv} to bound
\[ \| \sum_{n\leq N} e(\theta_n x) \cdot \big( \chi_k - \chi_{\varrho(k)} \big) *f_{\theta_n}(x_I) \|_{L^2_x(I)} \leq \mathbf{C} \cdot 2^{-v} \cdot |I|^{1/2} \]
uniformly for $k \in \Lambda_v(x_I)$. Above, $\{ \Lambda_v(x_I) : v \}$ are sets of times that are minimal with respect to the property that
\[ X(x_I) \subset \bigcup_{k \in \Lambda_v(x_I)} B_{\ell^2([N])}\big( \chi_k*\vec{f_{\Theta}}(x_I) , 2^{-v} \big),
\]
where $B_{\ell^2([N])}(\vec{v},r)$ is the ball of radius $r$ centered at $\vec{v} \in \ell^2([N])$ with respect to the sequence-space norm $\ell^2([N])$, and the parent function, $\varrho$, is as above.

The issue is the potential explosion
\[ \vec{N}_{2^{-v}}(x_I) \to \infty \text{ as } v \to \infty, \]
and there is no a priori way to rule out this enemy; if there were, there would be no logarithmic loss in \eqref{e-genRM}. The clever insight that Bourgain had that allowed him to push past this abstract issue was just Cauchy--Schwarz: we bound
\[ |\sum_{n\leq N} e(\theta_n x) \cdot \big( \chi_k - \chi_{\varrho(k)} \big) *f_{\theta_n}(x_I)|\leq N^{1/2} \cdot \big( \sum_{n \leq N} |\big( \chi_k - \chi_{\varrho(k)} \big) *f_{\theta_n}(x_I)|^2 \big)^{1/2} \leq \mathbf{C} \cdot N^{1/2} \cdot 2^{-v}
\]
uniformly for $k \in \Lambda_v(x_I)$, which yields the cheap bound
\begin{align*}
\| \max_{k \in \Lambda_v(x_I)} |\sum_{n\leq N} e(\theta_n x) \cdot \big( \chi_k - \chi_{\varrho(k)} \big) *f_{\theta_n}(x_I)| \|_{L^2_x(I)} \leq \mathbf{C} \cdot 2^{-v} \cdot N^{1/2} \cdot |I|^{1/2}.
\end{align*}
Altogether, Bourgain had obtained the bounds
\begin{multline*}
\| \max_{k \in \Lambda_v(x_I)} |\sum_{n\leq N} e(\theta_n x) \cdot \big( \chi_k - \chi_{\varrho(k)} \big) *f_{\theta_n}(x_I)| \|_{L^2_x(I)} \\ \leq \mathbf{C} \cdot 2^{-v} \cdot \min\{ \vec{N}_{2^{-v}}(x_I)^{1/2}, N^{1/2} \} \cdot |I|^{1/2},
\end{multline*}
see \eqref{e-L2bound}, which he cleverly interpolated, as per \eqref{e-interptrick},
\begin{align*} 2^{-v} \cdot \min\{ \vec{N}_{2^{-v}}(x_I)^{1/2}, N^{1/2} \} &\leq 
2^{-v} \cdot \vec{N}_{2^{-v}}(x_I)^{1/r} \cdot N^{1/2-1/r} \\
& \qquad \leq N^{1/2 - 1/r} \cdot \mathcal{V}^r_{\vec{f_{\Theta}}}(x_I) \leq \mathbf{C} \cdot \mathcal{V}^r_{\vec{f_{\Theta}}}(x_I) \; \; \; \; \; \; r = 2 + \frac{\mathbf{c}}{\log N}
\end{align*}
see \eqref{e-vrvec} and \eqref{e-ftheta},
obtaining a $v$-independent term on the right. Inserting this bound and arguing as in the heuristic analysis \eqref{2-e:jumps},
\begin{align}\label{e-atlast}
\eqref{e-L2close} &\leq \mathbf{C} \cdot \sum_{ v : 2^{-v} \leq \F_{\Theta}(x_I)} 2^{-v} \cdot \min\{ \vec{N}_{2^{-v}}(x_I)^{1/2}, N^{1/2} \} \cdot |I|^{1/2} \nonumber \\
&  \leq \mathbf{C} \sum_{v : 2^{-v} \leq \F_{\Theta}(x_I)/N^{1/2}} 2^{-v} \cdot N^{1/2} \cdot |I|^{1/2} + 
\mathbf{C} \sum_{v : \F_{\Theta}(x_I)/N^{1/2} \leq 2^{-v} \leq 2 \cdot \F_{\Theta}(x_I)} \mathcal{V}^{r}_{\vec{f_{\Theta}}}(x_I) \cdot |I|^{1/2} \nonumber \\
& \qquad \qquad \leq \mathbf{C} \cdot \F_{\Theta}(x_I) \cdot |I|^{1/2} + \mathbf{C} \cdot \log N \cdot \mathcal{V}^{r}_{\vec{f_{\Theta}}}(x_I) \cdot |I|^{1/2}, \; \; \; \; \; \; r = 2 + \frac{\mathbf{c}}{\log N}.
\end{align}
And, at last, after choosing $x_I$ carefully, we bound
\[ \| \M_{\Theta} f \|_{L^2(I)} \leq \mathbf{C} \cdot  \| \F_\Theta \|_{L^2(I)} + \mathbf{C} \cdot \log N  \cdot \| \mathcal{V}^r_{\vec{f_\Theta}} \|_{L^2(I)}, \]
which says that in a scale-$\mathbf{C}$, $L^2$-averaged sense, at all locations one has the following inequality
\[ \M_{\Theta} f \; \; \; \textrm{``}\leq\textrm{''} \; \; \; \F_\Theta + \log N \cdot \mathcal{V}^r_{\vec{f_\Theta}}, \; \; \; \; \; \;  r= 2 + \frac{\mathbf{c}}{\log N}.\]
In other words, up to logarithmic error, the vector valued maximal function and the vector-valued variation control $\M_{\Theta}$. And, square-summing over $\{ |I| = \mathbf{C} \}$, taking into account the related, convolution-based version of L\'{e}pingle's Inequality, Proposition \ref{p-LEP}, leads to the bound
\begin{align*}
\| \M_{\Theta} f \|_{L^2(\mathbb{R})} \leq \mathbf{C} \cdot( 1+ \log N \cdot \frac{r}{r - 2} ) \cdot \| f \|_{L^2(\mathbb{R})} \leq \mathbf{C} \cdot( 1+ \log^2 N) \cdot \| f\|_{L^2(\mathbb{R})},
\end{align*}
which satisfies \eqref{e-fexp}.

Guided by this intuition, we turn to the rigorous proof.

\section{The Proof of Proposition \ref{2-MAINPROP}, The Multi-Frequeny Maximal Inequality}

Motivated by our previous outline, we will seek to prove the following estimate:
\begin{align*}
\| \M_{\Theta} f \|_{L^2(\mathbb{R})} \leq \mathbf{C} \cdot \log^2 N \cdot \| f \|_{L^2(\mathbb{R})}.
\end{align*}
Accordingly we will restrict our attention only to scales $k \geq \mathbf{C} \log^2 N$, and just handle the complementary cases using a square function argument
\begin{multline*}
\| \sup_{\mathbf{C} \leq k \leq \mathbf{C} \cdot \log^2 N} |\Xi_k f|
\|_{L^2(\mathbb{R})} \leq \| \big( \sum_{\mathbf{C} \leq k \leq \mathbf{C}
  \cdot \log^2 N} |\Xi_k f|^2 \big)^{1/2} \|_{L^2(\mathbb{R})} \\ \leq \mathbf{C} \cdot \log N \cdot \sup_k \, \| \Xi_k f \|_{L^2(\mathbb{R})} 
  = \mathbf{C} \cdot \log N \cdot \sup_k \, \| \widehat{\Xi_k f}
  \|_{L^2(\mathbb{R})} \\ \leq \mathbf{C} \cdot \log N \cdot \sup_k \, \|\widehat{ \Xi_k } \|_{L^{\infty}(\mathbb{R})} \cdot \| \hat{f} \|_{L^2(\mathbb{R})} 
 \leq \mathbf{C} \cdot \log N \cdot \| f\|_{L^2(\mathbb{R})},
\end{multline*}
as
\[ \sup_k \sup_\xi |\widehat{\Xi_k}(\xi)| \leq 1, \]
see \eqref{e-fexp0}. We accordingly re-define $\F_{\Theta} = \F_{\Theta,\log^2 N}$, see \eqref{e-MAX}, and observe the inherited smoothness %which we will sometimes dominate by
%\[ G_{\theta}^2 \coloneqq  \sum_{n \leq N} M_{\geq} f_{\theta_n}^2 \]
%where
%\[ M_{\geq } f(x) \coloneqq  \sup_{k \geq \log^2 N} \, \frac{1}{2^k} \int_{-2^{k-1}}^{2^{k-1}} |f(x-t)| \ dt,\]
%is the Hardy--Littlewood Maximal Function restricted to large scales. 
%For future reference, we note the smoothness condition,
\begin{multline}\label{e-lip} 
|\F_{\Theta}(x) - \F_{\Theta}(y)|  \leq \mathbf{C} \cdot 
\sum_{|I| = 2^k \geq \log^2 N} \sum_{n \leq N} |\chi_k*f_{\theta_n}(x) - \chi_k*f_{\theta_n}(y) |  \\
  \leq \mathbf{C} \cdot N \cdot \sum_{k \geq \log^2 N} ( |x-y| \cdot 2^{-k} ) \cdot M_{\text{HL}} f(x) \leq \mathbf{C} \cdot |x-y| \cdot N^{-100} \cdot M_{\text{HL}} f(x), \end{multline}
(say), which says that $\F_{\Theta}$ is very smooth at scales $|I| = \mathbf{C}$. Above, we used the bound
\[ | \partial \chi_k(x)| \leq \mathbf{C} \cdot 2^{-k} \cdot 2^{-k} \cdot (1 + 2^{-k}\cdot |x|)^{-100}. \]
%
%\begin{align*}
%&\big| |\chi_k|*|f_{\theta_n}|(x) - |\chi_k|*|f_{\theta_n}|(y) \big| \leq \big( |\chi_k| - |\chi_k( \cdot - (x-y)| \big)* |f_{\theta_n}(x)| \\
%& \qquad \leq \mathbf{C} \cdot \big( |x-y| \cdot \| \big| \chi_k \big| \|_{\text{Lip}}  \big) \cdot 2^{-k} \cdot (1 + 2^{-k}| \cdot|)^{-100} *|f_{\theta_n}(x)| \\
%& \qquad \qquad \leq \mathbf{C} \cdot \big(|x-y| \cdot 2^{-k} \big) \cdot  2^{-k} \cdot (1 + 2^{-k}| \cdot|)^{-100} *|\chi|*|f(x)| \leq \mathbf{C} \cdot |x-y| \cdot 2^{-k}  \cdot M_{\text{HL}} f(x). 
%\end{align*}
This excision of scales allows us to be a little less delicate than Bourgain in making rigorous the heuristic  
\eqref{2-e:TAYLOR}: whereas Bourgain used a so-called \emph{best constant argument}, we will just use the following estimate, which is effective for small intervals relative to the scales $k \geq \mathbf{C} \cdot \log^2 N$:
\begin{align*}
\sum_{n \leq N} e(\theta_nx) \cdot \chi_k*f_{\theta_n} (x) &= \sum_{n \leq N} e(\theta_nx) \cdot \chi_k*f_{\theta_n}(x_I) + O\big( \frac{{N} \cdot |I|}{2^k} \cdot M_{\text{HL}} f (x) \big) \\
& \qquad = \sum_{n \leq N} e(\theta_nx) \cdot \chi_k*f_{\theta_n}(x_I) + O\big(2^{-k/2} \cdot M_{\text{HL}} f(x) \big)
\end{align*}
for any $x_I \in I$, certainly provided that $|I| \leq N^{\mathbf{C}}$.

In particular, for any $x \in I$, with $|I| = \mathbf{C}$, we may bound
\begin{align*}
\sup_{k \geq \mathbf{C} \cdot \log^2 N} |\Xi_k f(x)|  \leq \sup_{k \geq \mathbf{C} \cdot \log^2 N} \big| \sum_{n \leq N} e(\theta_nx) \cdot \chi_k*f_n(x_I) \big| + O\big( \sum_{k \geq \mathbf{C} \cdot \log^2 N} 2^{-k/2} \cdot M_{\text{HL}} f(x) \big),
\end{align*} so that for each $I$ we may bound
\begin{align*}
\| \sup_{k \geq \mathbf{C} \cdot \log^2 N} |\Xi_k f(x)| \|_{L^2(I)} &\leq \mathbf{C} \cdot \min_{x_I \in I} \| \sup_{k \geq \mathbf{C} \cdot \log^2 N} \big| \sum_{n \leq N} e(\theta_nx) \cdot \chi_k*f_n(x_I) \big| \|_{L^2(I)} \\
& \qquad + \mathbf{C} \cdot N^{-100} \cdot \| M_{\text{HL}} f \|_{L^2(I)}
\end{align*}
(say). Temporarily dropping the term involving $M_{\text{HL}}$ as inessential, we consider the first term
\begin{align*}
\| \sup_{k \geq \mathbf{C} \cdot \log^2 N} \big| \sum_{n \leq N} e(\theta_nx) \cdot \chi_k*f_n(x_I) \big| \|_{L^2(I)},
\end{align*}
which we bound using the entropic approach, see \eqref{e-atlast},
\begin{align*}
&\| \sup_{k \geq \mathbf{C} \cdot \log^2 N} \big| \sum_{n \leq N} e(\theta_nx) \cdot \chi_k*f_n(x_I) \big| \|_{L^2(I)} \\& \qquad \leq \mathbf{C} \cdot \big( \F_\Theta(x_I) \cdot |I|^{1/2} + \log N \cdot \mathcal{V}^r_{\vec{f_{\Theta}}}(x_I) \cdot |I|^{1/2} \big), \\
& \qquad \qquad \leq \mathbf{C} \cdot \big( \| \F_\Theta \|_{L^2(I)} + \log N \cdot \| \mathcal{V}^r_{\vec{f_{\Theta}}} \|_{L^2(I)}  + N^{-100} \cdot \| M_{\text{HL}} f \|_{L^2(I)} \big), 
\end{align*}
after choosing $x_I$ to minimize $\mathcal{V}^r_{\vec{f_{\Theta}}}$ on $I$, and using the smoothness
\[ \F_{\Theta}(x_I) = \F_{\Theta}(x) + O(N^{-100}  \cdot M_{\text{HL}}f(x) )\]
to bound
\[ \F_{\Theta}(x_I) \cdot |I|^{1/2} = \| \F_{\Theta} \|_{L^2(I)} + O\big( N^{-100} \cdot \| M_{\text{HL}} f \|_{L^2(I)} \big).\]
In particular, we have bounded 
\begin{align*}
\| \sup_{k \geq \mathbf{C} \cdot \log^2 N} | \Xi_k f | \|_{L^2(I)} \leq \mathbf{C} \cdot \big( \| \F_\Theta \|_{L^2(I)} + \log N \cdot \| \mathcal{V}^r_{\vec{f_{\Theta}}} \|_{L^2(I)}  + N^{-100} \cdot \| M_{\text{HL}} f \|_{L^2(I)} \big)
\end{align*}
where $r = 2 + \frac{\mathbf{c}}{\log N}$, so square-summing over $|I| = \mathbf{C}$ yields, at last, the bound
\begin{align*}
&\| \M_{\Theta} f \|_{L^2(\mathbb{R})} \leq \big( \sum_{|I| = \mathbf{C}}  \| \sup_{ k \geq \mathbf{C} \cdot \log^2 N} |\Xi_k f| \|_{L^2(I)}^2 \big)^{1/2} + \mathbf{C} \cdot \log N \cdot \| f\|_{L^2(\mathbb{R})} \\
&  \leq \mathbf{C} \cdot \big( \sum_{|I| = \mathbf{C}}  \| \F_\Theta \|_{L^2(I)}^2 \big)^{1/2} + \mathbf{C} \cdot \log N \cdot \big( \sum_{|I|= \mathbf{C}} \| \mathcal{V}^r_{\vec{f_{\Theta}}} \|_{L^2(I)}^2 \big)^{1/2} \\
&  \qquad \qquad + \mathbf{C} \cdot N^{-100} \cdot \big( \sum_{|I| = \mathbf{C}} \| M_{\text{HL}} f \|_{L^2(I)}^2 \big)^{1/2} + \mathbf{C} \cdot \log N \cdot \| f\|_{L^2(\mathbb{R})} \\
&  \qquad \leq \mathbf{C} \cdot \big( \| \F_{\Theta} \|_{L^2(\mathbb{R})} + \log N \cdot \| \mathcal{V}^r_{\vec{f_{\Theta}}} \|_{L^2(\mathbb{R})} + N^{-100} \cdot \| M_{\text{HL}} f \|_{L^2(\mathbb{R})} + \log N \cdot \| f\|_{L^2(\mathbb{R})} \big) \\
&  \qquad \qquad \leq \mathbf{C} \cdot \log^2 N \cdot \| f \|_{L^2(\mathbb{R})},
\end{align*}
completing the proof.

\section{Contemporary Work}

Since Bourgain's work, the topic of pointwise convergence of ergodic averages along polynomial orbits was taken up and greatly advanced by Mariusz Mirek, Eli Stein, and their collaborators
\parencite{[27],[28],[29],[30],[31]}, building on breakthrough work of \cite{[19]}. The current state of affairs was established in \cite{[30]}:
\begin{theo}\label{t:M}
Suppose that $(X,\mu)$ is a $\sigma$-finite measure space, $\tau: X \to X$ is a measure-preserving transformation, and $P \in \mathbb{Z}[\cdot]$ is a polynomial with integer coefficients. Then for each $1 < p < \infty, \ r > 2$
\begin{align*} 
&\| \mathcal{V}^r \big( \frac{1}{N} \sum_{n=1}^N \tau^{P(n)} f :N \big) \|_{L^p(X)} +
\sup_{ \lambda > 0} \, \| \lambda \cdot N_{\lambda} \big( \frac{1}{N} \sum_{n=1}^N \tau^{P(n)} f :N \big)^{1/2} \|_{L^p(X)} \\
& \qquad \leq \mathbf{C}_p \cdot ( 1+ \frac{r}{r-2}) \cdot \| f\|_{L^p(X)}.\end{align*}
\end{theo}
In other words, from a quantitative perspective, the rate of convergence of the abstract averages
\[ \frac{1}{N} \sum_{n=1}^N \tau^{P(n)} f \]
is precisely that of our entropic examples!

The key to this argument was a combinatorial partitioning of $\mathbb{Q} \cap [0,1]$ into the so-called Ionescu--Wainger exhaustion of the rationals: one replaces
\[ \mathcal{R}_s  \longrightarrow \mathcal{U}_s, \]
see \eqref{e:Rs},
where $\{ \mathcal{U}_s : s \}$ form a disjoint partition of $\mathbb{Q} \cap [0,1]$ which captures many of the analytical properties of $\mathcal{R}_s$, namely
\begin{align}\label{e:ud}
\sup_{A/Q \in \mathcal{U}_s} |S(A/Q)| \leq \mathbf{C}_\epsilon \cdot 2^{(\epsilon - 1/d)s},
\end{align}
but admit much more favorable arithmetic statistics, which allows
for the approximation
\begin{align*}
K_k' \; \; \; \textrm{``}&=\textrm{''} \; \; \; \sum_{s \leq \mathbf{c} \cdot k} \sum_{A/Q \in \mathcal{U}_s} S(A/Q) \cdot \text{Mod}_{A/Q} ( \chi_{(d - \mathbf{c})k}*\phi'_{dk}) = : \sum_{s \leq \mathbf{c} \cdot k} L_{k,s}
\end{align*}
to hold in $\ell^p(\mathbb{Z})$ as well.

Although Bourgain's entropic argument is less effective in general on $\ell^p, p \neq 2$, by applying the Rademacher--Menshov inequality and arguing as in \cite{[5]}, one is able to establish e.g.\ the estimate
\begin{align*}
   \| \sup_k |L_{k,s}*f| \|_{\ell^p(\mathbb{Z})} \leq \mathbf{C}_\epsilon \cdot 2^{ \epsilon s} \cdot 2^{ -\mathbf{c}_{p,d} s} \cdot \| f\|_{\ell^p(\mathbb{Z})}, \; \; \; \; \; \; 1 < p < \infty, \ \mathbf{c}_{p,d} < 1/d,
\end{align*}
and similarly for the jump-counting formulation. The loss in the number of frequencies is sub-exponential in $s$, as in the case of Bourgain's maximal function on $\ell^2$; the gain of 
\[ 2^{ -\mathbf{c}_{p,d} s} \]
follows from appropriately interpolating \eqref{e:ud}.

This quantitative improvement over the sharpest estimates for $\sup_k |L_{k,s}'*f|$,
\[ \| \sup_k |L'_{k,s}*f| \|_{\ell^p(\mathbb{Z})}\leq \mathbf{C}_{\epsilon,p} \cdot 2^{(\epsilon + 1) s} \cdot 2^{ -\mathbf{c}_{p,d} s} \cdot \| f\|_{\ell^p(\mathbb{Z})}, \; \; \; 1 < p < \infty, \ \mathbf{c}_{p,d} < 1/d, \]
speaks to the flexibility of these arguments, which indeed extend to handle the case of the $r$-variation and jump-counting operators.

%% printbibliography is the command from the package biblatex
%\printbibliography


\begin{thebibliography}{9}

\bibitem{[1]}
A. Bellow. \emph{Two Problems}. Lecture Notes in Math. 945, Springer-Verlag, Berlin,
pp. 429-431. 

\bibitem{[2]}
A. Bellow; V. Losert. \emph{On sequences of density zero in ergodic theory}, Contemp.
Math. 26 (1984), 49-60. 

\bibitem{[3]}
G. Birkhoff. \emph{Proof of the ergodic theorem.} Proc Natl Acad Sci USA 17 (12): 656-660
(1931)


\bibitem{[4]}
J. Bourgain.
\emph{On high-dimensional maximal functions associated to convex bodies.}
Amer. J. Math. 108 (1986), no. 6, 1467–1476.

\bibitem{[5]}
J. Bourgain. 
\emph{ On the maximal ergodic theorem for certain subsets of the positive integers. }
Israel J. Math. 61 (1988), 39-72.

\bibitem{[6]}
J. Bourgain. 
\emph{ On the pointwise ergodic theorem on $L^p$ for arithmetic sets. }
Israel J. Math. 61 (1988), no. 1, 73-84.

\bibitem{[7]}
J. Bourgain. \emph{
Almost sure convergence and bounded entropy. }
Israel J. Math. 63 (1988), no. 1, 79–97.

\bibitem{[8]}
J. Bourgain.
\emph{Bounded orthogonal systems and the $\Lambda(p)$-set problem.}
Acta Math. 162 (1989), no. 3-4, 227–245.

\bibitem{[9]}
J. Bourgain. 
\emph{ Pointwise ergodic theorems for arithmetic sets. }
Inst. Hautes \'{E}tudes Sci. Publ. Math. (69):5-45, 1989.
With an appendix by the author, Harry Furstenberg, Yitzhak Katznelson and Donald S. Ornstein.

\bibitem{[10]}
J. Bourgain.
\emph{Double recurrence and almost sure convergence. }
J. Reine Angew. Math. 404 (1990), 140–161.

\bibitem{[11]}
Z. Buczolich; D. Mauldin. \emph{Concepts behind divergent ergodic averages along the squares.} Ergodic Theory
and Related Fields, AMS, Contemporary Mathematics Vol. 430 (2007) 41-56


\bibitem{[12]}
A. Calder\'{o}n. \emph{ Ergodic theory and translation invariant operators.} 
Proc. Nat. Acad. Sci., USA
59 (1968), 349-353

\bibitem{[13]}
C. Demeter. \emph{
Pointwise convergence of the ergodic bilinear Hilbert transform. }
Illinois J. Math. 51 (2007), no. 4, 1123–1158.

\bibitem{[14]}
C. Demeter; M. Lacey; T. Tao; C. Thiele. \emph{ 
Breaking the duality in the return times theorem. }
Duke Math. J. 143 (2008), no. 2, 281–355.

\bibitem{[15]}
C. Demeter; T. Tao; C. Thiele. \emph{ 
Maximal multilinear operators. }
Trans. Amer. Math. Soc. 360 (2008), no. 9, 4989–5042. 

\bibitem{[16]}
H. Furstenberg. Proc. Durham Conf., June 1982

\bibitem{[17]}
S. Guo; J. Roos; P.-L. Yung.
\emph{Sharp variation-norm estimates for oscillatory integrals related to Carleson's theorem. } Anal. PDE 13 (2020), no. 5, 1457–1500.


\bibitem{[18]}
L.-K. Hua. \emph{Introduction to number theory.} Springer-Verlag, Berlin-New York, 1982. Translated from the Chinese by Peter Shiu.

\bibitem{[19]}
A. Ionescu; S. Wainger.
\emph{$L^p$ boundedness of discrete singular Radon transforms.}
{{J. Amer. Math. Soc.} {19}, (2005), no. 2, 357–-383.}



\bibitem{[20]}
R. Jones; R. Kaufman; J. Rosenblatt; M. Wierdl. \emph{Oscillation in ergodic theory.}  Ergodic Theory Dynam. Systems  18  (1998),  no. 4, 889-935.

\bibitem{[21]}
R. Jones; A. Seeger; J. Wright. \emph{Strong variational and jump inequalities in harmonic analysis}.
Trans. Amer. Math. Soc. 360 (2008), no. 12, 6711-6742.


\bibitem{[22]}
R. Jones; G. Wang. \emph{Variation inequalities for the F\'{e}jer and Poisson kernels}. Trans. Amer. Math. Soc., 356 (2004), 4493-4518.

\bibitem{[23]}
B. Krause. \emph{Discrete Analogues in Harmonic Analysis: Bourgain, Stein, and Beyond.}
To appear as:\\
Graduate Studies in Mathematics, 224. American Mathematical Society, Providence, RI, 2023


\bibitem{[24]}
M. Lacey.
\emph{ The bilinear maximal functions map into $L^p$ for $\frac{2}{3} < p \leq 1$.} Ann. of Math. (2) 151 (2000), no. 1, 35–57.

\bibitem{[25]}
P. LaVictoire. \emph{ Universally $L^1$-bad arithmetic sequences.}
J. Anal. Math. 113 (2011), 241–263.

\bibitem{[26]}
D. L\'{e}pingle. \emph{La variation d'ordre $p$ des semi-martingales.}
Z.F.W. 36, 1976, 295-316.

\bibitem{[27]}
M. Mirek.
\emph{Square function estimates for discrete Radon transforms.}
Anal. PDE 11 (2018), no. 3, 583–608.

\bibitem{[28]}
M. Mirek; E. Stein; B. Trojan. \emph{$\ell^p(\mathbb{Z}^d)$-estimates for discrete operators of Radon type: Maximal functions and vector-valued estimates.} J. Funct. Anal. 277 (2019), no. 8, 2471–2521.

\bibitem{[29]}
M. Mirek; E. Stein; B. Trojan. \emph{$\ell^p(\mathbb{Z}^d)$-estimates for discrete operators of Radon type: Variational estimates.}
Invent. Math. 209 (2017), no. 3, 665–748.




\bibitem{[30]}
M. Mirek; E. Stein; P. Zorin-Kranich.
\emph{Jump Inequalities for Translation-Invariant Operators
of Radon Type on $\mathbb{Z}^d$.} Adv. Math. 365 (2020), 107065

\bibitem{[31]}
M. Mirek; B. Trojan. \emph{Discrete maximal functions in higher dimensions and applications to ergodic theory.}
Amer. J. Math. 138 (2016), no. 6, 1495–1532.





\bibitem{[32]}
R. Oberlin; A. Seeger; T. Tao; C. Thiele; J. Wright. \emph{
A variation norm Carleson theorem.}
J. Eur. Math. Soc. (JEMS) 14 (2012), no. 2, 421–464.

\bibitem{[33]}
J. Qian. \emph{The $p$-variation of partial sum processes and the empirical process}. Ann. of Prob. Theory, 77 (1988), 1370-1383.

\bibitem{[34]}
E. Stein.
\emph{Harmonic analysis: real-variable methods, orthogonality, and oscillatory integrals.}
Princeton Mathematical Series, 43. Monographs in Harmonic Analysis, III. Princeton University
Press, Princeton, NJ, 1993.

\bibitem{[35]}
T. Tao. \emph{ 
Exploring the toolkit of Jean Bourgain. }
Bull. Amer. Math. Soc. (N.S.) 58 (2021), no. 2, 155–171.

\bibitem{[36]}
J.-P. Thouvenot.\emph{ Almost sure convergence of ergodic means along some subsequences of integers (after Jean Bourgain)}
S\'{e}minaire Bourbaki, Vol. 1989/90.
Ast\'{e}risque No. 189-190 (1990), Exp. No. 719, 133–153.

\end{thebibliography}
\end{document}

%%% Local Variables:
%%% mode: latex
%%% TeX-master: t
%%% End: